\documentclass[a4paper,reqno,11pt]{amsart}
\usepackage{a4wide}

\usepackage[all]{xy}
\usepackage{tikz}
\usetikzlibrary{patterns,shapes,decorations.pathmorphing,decorations.pathreplacing,calc,arrows}
\usepackage{verbatim}
\usepackage{amstext}
\usepackage{amsmath}
\usepackage{amscd}
\usepackage{amsthm}
\usepackage{amsfonts}
\usepackage{enumerate}
\usepackage{graphicx}
\usepackage{latexsym}
\usepackage{mathrsfs}
\usepackage{hyperref}
\usepackage{mathtools}
\xyoption{all}

\usepackage{pstricks}
\usepackage{lscape}
\usepackage{comment}

\newtheorem{theorem}{Theorem}[section]
\newtheorem*{theorem*}{Theorem}

\newtheorem{corollary}[theorem]{Corollary}
\newtheorem{lemma}[theorem]{Lemma}
\newtheorem{proposition}[theorem]{Proposition}
\newtheorem{definition-proposition}[theorem]{Definition-Proposition}
\newtheorem{problem}[theorem]{Problem}
\newtheorem{question}[theorem]{Question}

\theoremstyle{definition}
\newtheorem{definition}[theorem]{Definition}

\newtheorem{remark}[theorem]{Remark}
\newtheorem{example}[theorem]{Example}

\newcommand{\CC}{\mathcal{C}}
\newcommand{\CCC}{\mathsf{C}}
\newcommand{\DDD}{\mathsf{D}}

\newcommand{\KKK}{\mathsf{K}}

\newcommand{\TT}{\mathcal{T}}
\newcommand{\UU}{\mathcal{U}}

\newcommand{\Z}{\mathbb{Z}}

\newcommand{\Sbb}{\mathbb{S}}

\newcommand{\bo}{\operatorname{b}\nolimits}

\newcommand{\projdim}{\operatorname{projdim}\nolimits}
\newcommand{\gldim}{\operatorname{gldim}\nolimits}

\newcommand{\Ext}{\operatorname{Ext}\nolimits}
\newcommand{\Out}{\operatorname{Out}\nolimits}\newcommand{\Pic}{\operatorname{Pic}\nolimits}

\newcommand{\Hom}{\operatorname{Hom}\nolimits}
\newcommand{\rad}{\operatorname{rad}\nolimits}
\newcommand{\End}{\operatorname{End}\nolimits}
\newcommand{\Aut}{\operatorname{Aut}\nolimits}
\newcommand{\injdim}{\operatorname{injdim}\nolimits}
\newcommand{\op}{\operatorname{op}\nolimits}
\newcommand{\sg}{\operatorname{sg}\nolimits}
\newcommand{\RHom}{\mathbf{R}\strut\kern-.2em\operatorname{Hom}\nolimits}
\newcommand{\Lotimes}{\mathop{\stackrel{\mathbf{L}}{\otimes}}\nolimits}

\DeclareMathOperator{\moduleCategory}{\mathsf{mod}} \renewcommand{\mod}{\moduleCategory}
\DeclareMathOperator{\Mod}{\mathsf{Mod}}
\DeclareMathOperator{\proj}{\mathsf{proj}}
\DeclareMathOperator{\inj}{\mathsf{inj}}
\DeclareMathOperator{\ind}{\mathsf{ind}}
\newcommand{\Db}{\DDD^{\bo}}
\newcommand{\Kb}{\KKK^{\bo}}

\DeclareMathOperator{\per}{\mathsf{per}}

\DeclareMathOperator{\add}{\mathsf{add}}

\DeclareMathOperator{\lcm}{lcm}

\DeclareMathOperator{\CYdim}{CY-dim}

\newcommand{\stmod}{\mathop{\underline{\mod}}\nolimits}

\newcommand\cx{\mathop{\rm cx}\nolimits}

\newcommand{\cut}{\ar@{-}@[|(5)]}

\numberwithin{equation}{section}

\title[Periodic trivial extension algebras and fractionally Calabi--Yau algebras]{Periodic trivial extension algebras and\\ fractionally Calabi--Yau algebras}

\dedicatory{Dedicated to the memory of Andrzej Skowro\'{n}ski}

\author[Chan]{Aaron Chan}
\address[Chan]{Graduate School of Mathematics, Nagoya University, Furocho, Chikusaku, Nagoya 464-8602, Japan}

\author[Darp\"o]{Erik Darp\"o}
\address[Darp\"o]{Department of Mathematics, Link\"oping University, SE-58183 Link\"oping, Sweden}

\author[Iyama]{Osamu Iyama}
\address[Iyama]{Graduate School of Mathematical Sciences, University of Tokyo, 3-8-1 Komaba Meguro-ku, Tokyo 153-8914, Japan}

\author[Marczinzik]{Ren\'{e} Marczinzik}
\address[Marczinzik]{Institute of algebra and number theory, University of Stuttgart, Pfaffenwaldring 57, 70569 Stuttgart, Germany}

\begin{document}

\vspace{1em}

\begin{abstract}
We study periodicity and twisted periodicity of the trivial extension algebra $T(A)$ of a finite-dimensional algebra $A$. Our main results show that (twisted) periodicity of $T(A)$ is equivalent to $A$ being (twisted) fractionally Calabi--Yau of finite global dimension. We also extend this result to a large class of self-injective orbit algebras.
As a significant consequence, these results give a partial answer to the periodicity conjecture of Erdmann--Skowro\'nski, which expects the classes of periodic and twisted periodic algebras to coincide.
On the practical side, it allows us to construct a large number of new examples of periodic algebras and fractionally Calabi--Yau algebras. 
We also establish a connection between periodicity and cluster tilting theory, by showing that twisted periodicity of $T(A)$ is equivalent the $d$-representation-finiteness of the $r$-fold trivial extension algebra $T_r(A)$  for some $r,d\ge 1$. This answers a question by Darp\"o and Iyama.

As applications of our results, we give answers to some other open questions.
We construct periodic symmetric algebras of wild representation type with arbitrary large minimal period, answering a question by Skowro\'nski.
We also show that the class of twisted fractionally Calabi--Yau algebras is closed under derived equivalence, answering a question by Herschend and Iyama.
\end{abstract}

\maketitle 

\setcounter{tocdepth}{1}
\tableofcontents

\section{Introduction}

The syzygy functor $\Omega_A:\stmod A\to\stmod A$ \cite{AB} is a fundamental tool in the homological algebra of a Noetherian ring $A$.
A finite-dimensional $k$-algebra $A$ is said to be \emph{periodic} (of period $n$) if $\Omega_{A^{\mathrm{e}}}^n(A) \simeq A$ as $A^{\mathrm{e}}$-modules for some $n \geq 1$, where $A^{\mathrm{e}}=A \otimes_k A^{\op}$ is the enveloping algebra of $A$.
This implies the periodicity of the syzygy functor $\Omega_A^n$, and is often a more workable condition than the latter.
%
All periodic algebras are self-injective, and amongst self-injective algebras, the periodic ones constitute a fundamental subclass, with many important properties.
The following problem, raised in \cite[Problem~1]{ES}, is central to understanding self-injective algebras, and also highly significant in the theory of Hochschild cohomology and support varieties \cite{GSS,EH}.

\begin{problem} \label{periodicityquestion}
For a self-injective algebra $B$, when is $B$ periodic? 
\end{problem}

For example, preprojective algebras of Dynkin type are periodic. Also, the trivial extension algebra of the path algebra $kQ$ of an acyclic quiver $Q$ is periodic if and only if $Q$ is Dynkin \cite{BBK}. Periodic algebras appear also in group representation theory, topology and algebraic geometry, e.g., some contraction algebras are periodic \cite{DW}. We refer to \cite{ES} for a survey on periodic algebras, and to e.g.\ \cite{AS1,BES,Du2,Du3,ES2,ES3,JKM} for more recent contributions.

There are some natural variations of periodicity, see Figure~\ref{dgm:implications} in Section~\ref{sec:periodic main}. In particular, we call a finite-dimensional $k$-algebra $A$ \emph{twisted periodic} if $\Omega_{A^{\mathrm{e}}}^n(A) \simeq {}_1 A_\phi$ as $A^{\mathrm{e}}$-modules, for some $n\geq 1$ and $k$-algebra automorphism $\phi$ of $A$ (here, ${}_1A_\phi$ denotes the $A^{\mathrm{e}}$-module $A$ with right action twisted by $\phi$).
In the recent article \cite{ES2} the following important question is formulated as a conjecture, and given an affirmative answer for group algebras.

\begin{question}[Periodicity conjecture \cite{ES2}]\label{conj:periodicity}
Is every finite-dimensional twisted periodic algebra periodic? 
\end{question}

One of the most fundamental classes of self-injective algebras is given by the \emph{trivial extension algebra} $T(A)$ of a finite-dimensional $k$-algebra $A$.
Recall that $T(A) = A \oplus \Hom_k(A,k)$ with multiplication given by $(a,f)(b,g) = (ab, ag+fb)$. Trivial extension algebras,
together with their dg (=differential graded) analogues, have played important roles in the representation theory of algebras. Examples include the classification of ($d$-)representation-finite and tame self-injective (dg) algebras \cite{T,HW,R,Sk,SY1,DI,J}, the study of derived categories \cite{Hap} and cluster categories \cite{K1,Am}, gentle algebras and Brauer graph algebras \cite{AS2,Sc}.
They also appear in other fields, such as symplectic and contact geometry \cite{KS,EL}.

The purpose of this paper is to study periodicity and twisted periodicity of the trivial extension algebra $T(A)$, including Problem~\ref{periodicityquestion} and Question~\ref{conj:periodicity}, and relate it to homological properties of the algebra $A$.
We will give a complete solution to Problem~\ref{periodicityquestion} in this vein for $B=T(A)$. As an application, we give a large number of new examples of periodic algebras, including many of wild representation type.
In Section~\ref{section:selfinjective}, we extend our solution  to a large class of self-injective algebras given by the repetitive category \ \cite{Sk,SY1,SY3,ES}.

To explain our solution to Problem~\ref{periodicityquestion}, recall that the bounded derived category $\Db(\mod A)$ of a finite-dimensional algebra $A$ of finite global dimension has a Serre functor $\nu$, see \eqref{define nu} below. Such an algebra $A$ is said to be \emph{fractionally Calabi--Yau} (also \emph{$\frac{m}{\ell}$-Calabi--Yau}) if there exist integers $\ell>0$ and $m$ such that $\nu^\ell$ and $[m]$ are isomorphic as functors on $\Db(\mod A)$ \cite[\S18.6]{Ye}.
For example, the path algebra of a Dynkin quiver is $\frac{h-2}{h}$-Calabi--Yau, where $h$ is the Coxeter number \cite{MY} -- see Theorem~\ref{Dynkin is CY} for a more precise statement.
There is also a weaker notion of \emph{twisted} fractionally Calabi--Yau, in which the defining isomorphism of functors is taken up to a twist by an algebra automorphism.
There are many important examples of (twisted) fractionally Calabi--Yau algebras, in representation theory \cite{Gra,HI,Lad,R,Y} as well as in algebraic geometry \cite{GL,KLM,FK,K,HIMO}. They play important roles in various areas, e.g., integrable systems \cite{K3,IIKNS}, Hochschild cohomology \cite{Per} and mathematical physics \cite{CaCe}.

Our first main result gives a solution to Problem~\ref{periodicityquestion} for trivial extension algebras.

\begin{theorem}[Corollaries~\ref{finite out} and \ref{application to orbit algebra 2}]\label{main theorem}
Let $A$ be a finite-dimensional algebra over a field $k$ such that $A/\rad A$ is a separable $k$-algebra (e.g., when $k$ is perfect).
Then the following conditions are equivalent.
\begin{enumerate}[\rm(i)]
\item $T(A)$ is periodic.
\item $A$ has finite global dimension and is fractionally Calabi--Yau.
\end{enumerate}
Moreover, let $G$ be an admissible group of automorphisms of the repetitive category $\widehat{A}$ (see Section \ref{subsec:trivext}) containing $\nu_{\widehat{A}}^\ell$ for some $\ell\ge1$. Then the following condition is equivalent to (i) and (ii).
\begin{enumerate}[\rm(i)]
\item[\rm(iii)] $\widehat{A}/G$ is periodic.
\end{enumerate}
\end{theorem}

This gives a large number of new periodic algebras, see Section~\ref{sec:eg}. As a consequence, we get a conceptual proof of the periodicity of the trivial extension algebras of the path algebras of Dynkin quivers mentioned above, see Example~\ref{T(kQ) is periodic}.
The trickiest part of the proof of Theorem \ref{main theorem} is the implication (ii)$\Rightarrow$(i), which will be shown in Section \ref{sec:fCY to per} by using the relative bar resolution of a certain differential graded algebra quasi-isomorphic to $T(A)$.

The classification of periodic symmetric algebras of tame representation type is a well studied subject, e.g.\ \cite{ES3,ES4,ES5}. According to an Oberwolfach talk in January 2020 by Skowro\'{n}ski \cite{S}, there is no known example of a family of wild symmetric algebras with unbounded minimal periods.
As an application of our results, we can construct many such examples.
For example, the trivial extension $T(A)$ of the incidence algebra $A$ of the Boolean lattice with $2^n$ elements has minimal period $3+n$ if $n$ is odd or the characteristic of $k$ is two, and $2(n+3)$ otherwise (Corollary~\ref{periodoftp}). Here $T(A)$ is indeed wild for $n \geq 4$.

\medskip
In our second main result, we give several characterisations of \emph{twisted} periodicity for $T(A)$.
Recall that, for a positive integer $d$, a finite-dimensional algebra $A$ is said to be \emph{$d$-representation-finite} if there exists a $d$-cluster-tilting $A$-module. 
For algebras of finite global dimension, this is closely related to the notion of twisted fractionally Calabi--Yau \cite{HI} and, for self-injective algebras, to periodicity \cite{EH}. Using results from \cite{DI}, we characterize twisted periodicity of $T(A)$ via $d$-representation-finiteness of the \emph{$r$-fold} trivial extension algebra $T_r(A)$ (see Section~\ref{subsec:trivext}).
Our second main result can be summarized as follows. 

\begin{theorem}[Theorem~\ref{main theorem 2 again}, Corollary~\ref{application to orbit algebra 2}]\label{main theorem 2}
Let $A$ be a finite-dimensional algebra over a field $k$ such that $A/\rad A$ is a separable $k$-algebra. The following conditions are equivalent.
\begin{enumerate}[\rm(i)] 
\item $T(A)$ is twisted periodic.
\item Each $T(A)$-module has complexity at most one.
\item There exist $d,r\ge1$ such that $T_r(A)$ is $d$-representation-finite.
\item $A$ has finite global dimension and is twisted fractionally Calabi--Yau.
\end{enumerate}
Moreover, let $G$ be an admissible group of automorphisms of $\widehat{A}$. Then the following conditions are equivalent to (i)-(iv).
\begin{enumerate}[\rm(v)] 
\item $\widehat{A}/G$ is twisted periodic.
\item[\rm(vi)] Each $\widehat{A}/G$-module has complexity at most one.
\end{enumerate}
\end{theorem}

Together with Theorem~\ref{main theorem}, this establishes the following diagram of implications.
\begin{align*}
    \vcenter{
\xymatrix@C=60pt{
{\phantom{abcd}A:\text{ fractionally CY}\phantom{graddd}} \ar@{=>}[d]^{\text{trivial}\phantom{abcd}} \ar@{<=>}[r]^(.55){\text{Thm }\ref{main theorem}} & {\phantom{abcd}T(A):\text{ periodic}\phantom{grad}} \ar@{=>}[d]^{\text{trivial}}  \\
{\quad\;\; A:\text{ twisted fractionally CY}\qquad}  \ar@{<=>}[r]^(.55){\text{Thm }\ref{main theorem 2}}& {\phantom{ged}T(A):\text{ twisted periodic}}
}}
\end{align*}
We remark that the implication (iii)$\Rightarrow$(iv) in Theorem~\ref{main theorem 2} gives a positive answer to \cite[Question 6.1(2)]{DI}.

\medskip
We apply our Theorems~\ref{main theorem} and ~\ref{main theorem 2} to study the periodicity conjecture (Question~\ref{conj:periodicity}) for the trivial extension algebra $T(A)$ (and orbit algebras $\widehat{A}/G$ more generally) in terms of the much simpler algebra $A$. In particular, we get the following result.

\begin{corollary}[Corollaries \ref{finite out}, \ref{application to orbit algebra 2}]\label{finite out again}
Let $A$ be a finite-dimensional algebra over a field $k$ such that $A/\rad A$ is a separable $k$-algebra.
Let $B = T(A)$ or, more generally, $B=\widehat{A}/G$ for an admissible group $G$ of automorphisms of $\widehat{A}$ containing $\nu_{\widehat{A}}^{\ell}$ for some $\ell\ge1$. If the outer automorphism group of $A$ is finite, then $B$ is periodic if and only if it is twisted periodic.
\end{corollary}

This result implies that the periodicity conjecture is true for the trivial extension algebra of the incidence algebra of any finite bounded poset (Theorem \ref{periodicity conjecture for poset}).
More generally, our results reduce the periodicity conjecture for trivial extensions to the following general question for algebras of finite global dimension, posed in \cite{HI}.

\begin{question}\cite{HI}\label{remove twist 2}
Let $A$ be a finite-dimensional $k$-algebra of finite global dimension that is twisted fractionally Calabi--Yau. Is $A$ fractionally Calabi--Yau?
\end{question}

In fact, the periodicity conjecture for trivial extension algebras is equivalent to Question~\ref{remove twist 2}, which ought to be more accessible in most cases.

\begin{corollary}
Let $k$ be a perfect field. Then Question~\ref{remove twist 2} has an affirmative answer if and only if the periodicity conjecture holds for all trivial extension algebras of finite-dimensional $k$-algebras of finite global dimension.
\end{corollary}

Another application of Theorem~\ref{main theorem 2} is the following result, which gives a positive answer to a question posed in \cite[Remark 1.6(c)]{HI}.

\begin{corollary}[Corollary \ref{derived closed}]
Let $k$ be a perfect field. Then the class of twisted fractionally Calabi--Yau $k$-algebras of finite global dimension is closed under derived equivalence.
\end{corollary}

Our study motivated us to summarize in Section \ref{sec:eg} various examples of fractionally Calabi--Yau algebras (of finite global dimension) from the literature.  Along the way, we give some new constructions of new fractionally Calabi--Yau algebras. One of them is by simply taking tensor products of fractionally Calabi--Yau algebras (Proposition \ref{prop:cydim}).  Another one reveals yet another connection to cluster tilting theory.

\begin{theorem}[Theorem~\ref{stable Auslander is CY}]
Let $A$ be a $d$-representation-finite algebra with $\gldim A\le d$, $M$ its unique basic $d$-cluster-tilting $A$-module, and $E:=\underline{\End}_A(M)$ the stable $d$-Auslander algebra.
Then the algebra $E$ is twisted fractionally Calabi--Yau, and so $T(E)$ is twisted periodic.  Moreover, if $E$ is, in addition, (untwisted) fractionally Calabi--Yau, then $T(E)$ is periodic.
\end{theorem}

Lastly, as an application of Theorem~\ref{main theorem}, one can find new examples of fractionally Calabi--Yau algebras by using a computer algebra system, such as \cite{QPA}, to check whether the trivial extension of a candidate algebra is periodic.
We illustrate this in Example \ref{11 points} by sketching a classification of fractionally Calabi--Yau incidence algebras of distributive lattices on 11 points, which leads to the discovery of new fractionally Calabi--Yau algebras.

In view of Question~\ref{remove twist 2} and partial results such as Corollary~\ref{finite out again}, we find it natural to pose the following question.

\begin{question}
  Does every (twisted) fractionally Calabi--Yau algebra of finite global dimension have a finite outer automorphism group?
\end{question}

Note that we cannot drop the assumption of finite global dimension in the question above. In fact, any self-injective algebra is twisted $\frac{0}{1}$-Calabi--Yau whose twist is given by the Nakayama automorphism. But such an algebra is usually not fractionally Calabi--Yau since the Nakayama automorphism often has infinite order.

This article is structured as follows.
In Section~\ref{section 1} we give preliminary results, and in Section~\ref{section 2} we summarize known results on (twisted) periodicity of algebras.
Section~\ref{sec:periodic main} features the proof of one of our main result, Theorem~\ref{main theorem 2}, for trivial extension algebras. Using preliminary results given in Section~\ref{section 4}, we prove the trivial extension case of our second main result, Theorem~\ref{main theorem}, in Section~\ref{sec:fCY to per}.
Section~\ref{section:selfinjective} concludes the proofs of the two main theorems, by extending our results about trivial extensions to more general classes of orbit algebras.
Various examples of (twisted) fractionally Calabi--Yau algebras and (twisted) periodic algebras are discussed in Section~\ref{sec:eg}. Finally, in the appendix Section~\ref{sec:DynkinCY}, we give explicit Calabi--Yau dimensions of the path algebras of Dynkin quivers.

\section{Preliminaries}\label{section 1}

\subsection{Conventions and basic facts}\label{section 1.1}

Throughout this paper, $k$ denotes a field and $A$ a finite-dimensional $k$-algebra.
Unless otherwise specified, by $A$-module we mean finitely generated right $A$-module.
The category of $A$-modules is denoted by $\mod A$.
The \emph{stable module category} $\stmod A$ of $A$ is the quotient of $\mod A$ by the ideal of morphisms factoring through a projective.
We denote by $D:=\Hom_k(-,k)$ the $k$-linear duality, by $Z(A)$ the centre of $A$, and by $A^\times$ the group of unit elements of $A$.
For general background on representation theory and homological algebra of finite-dimensional algebras, we refer for example to \cite{ASS,SY2,Z}.

If $A$ is graded by some group $G$, the category of $G$-graded $A$-modules is denoted by $\mod^{G}A$, and the corresponding stable module category by $\stmod^{G}A$. For $M\in\mod^GA$, recall that the syzygy $\Omega(M)=\Omega_{A}(M)$ is the kernel of the projective cover of $M$ in $\mod^GA$.
If $A$ is a self-injective algebra, then $\stmod^GA$ has the structure of a triangulated category, and $\Omega:\stmod^GA\to\stmod^GA$ gives the inverse suspension functor $[-1]$.
For $a\in G$, we denote by $(a):\mod^GA\to\mod^GA$ the $a$-th grading shift functor, given by $(M(a))_i = M_{i+a}$ for each $i\in G$.

In this paper, we assume the grading group to be the integers, unless otherwise stated.
Note that, for any integer $n$, a $\Z$-graded algebra can be regarded as an $(\Z/n\Z)$-graded algebra in a canonical way, and thus there is a forgetful functor $\mod^{\Z}A\to\mod^{\Z/n\Z}A$.

For a $k$-algebra $A$, we denote by $\Aut_k(A)$ the group of $k$-algebra automorphisms of $A$, and by $\Out_k(A) = \Aut_k(A)/\mathop{\rm Inn}\nolimits_k(A)$ the outer automorphism group. Recall that an inner automorphism $\phi\in \mathop{\rm Inn}\nolimits_k(A)$ is an automorphism given by $\phi(a) = bab^{-1}$ for some $b\in A^{\times}$.
%
 If $A$ is graded, $\Aut_k^{\Z}(A)$ denotes the graded automorphism group (consisting of all grading-preserving automorphisms) of $A$.
We write $\phi^*:\mod A\to\mod A$ (respectively, $\phi^*:\mod^{\Z}A\to\mod^{\Z}A$) for the restriction functor along an automorphism $\phi\in\Aut_k(A)$ (respectively, $\phi\in \Aut_k^{\Z}(A)$). 

For $\phi\in\Aut_k(A)$, we denote by ${}_1A_\phi$ the $A^{\mathrm{e}}$-module, where the right action is twisted by $\phi$. Then $\phi^*$ is given by the tensor functor $-\otimes_A({}_1A_\phi)$. We also use the functor $\phi_*:=-\otimes_A({}_\phi A_1)$.
Let $\Pic_k(A)$ be the Picard group of $A$ \cite{Z}. Thus, an element of $\Pic_k(A)$ is the isomorphism class of an $A^{\mathrm{e}}$-module $X$ for which there exists an $A^{\mathrm{e}}$-module $Y$ such that $X\otimes_AY\simeq A\simeq Y\otimes_AX$ as $A^{\mathrm{e}}$-modules, and the multiplication is given by the tensor product.

The following is elementary.

\begin{proposition}\label{Out Pic}
For $\phi,\psi\in\Aut_k(A)$, the following conditions are equivalent.
\begin{enumerate}[\rm(i)]
\item $\phi=\psi$ in $\Out_k(A)$.
\item ${}_1A_\phi\simeq{}_1A_\psi$ as $A^{\mathrm{e}}$-modules.
\item The functors $\phi^*,\psi^*:\mod A\to\mod A$ are isomorphic.
\end{enumerate}
If $A$ is basic, then there is an isomorphism $\Out_k(A)\simeq\Pic_k(A)$ given by $\phi\mapsto{}_1A_\phi$.
\end{proposition}

It is elementary that any two complete sets of orthogonal primitive idempotents of $A$ are conjugate of each other \cite[Theorem 3.4.1]{DK}. In particular, for each complete set $e_1,\ldots,e_n$ of orthogonal primitive idempotents of $A$ and each $\phi\in\Aut_k(A)$, there exists $\psi\in\Aut_k(A)$ and a permutation $\sigma\in\mathfrak{S}_n$ such that $\phi=\psi$ in $\Out_k(A)$ and $\psi(e_i)=e_{\sigma(i)}$ for each $1\le i\le n$.

Let us recall the following properties of graded algebras.

\begin{proposition}[\cite{GG1,GG2}]\label{forgetful}
Let $A$ be a graded algebra, and $F:\mod^{\Z}A\to\mod A$ the forgetful functor.
\begin{enumerate}[\rm(a)]
\item $\rad A$ is a homogeneous ideal of $A$, and every simple $A$-module is gradable. \label{forget1}
\item $F$ sends simple objects in $\mod^{\Z}A$ to simple objects in $\mod A$. \label{forget2}
\item If $f:P\to M$ is a projective cover in $\mod^{\Z}A$, then $F(f):F(P)\to F(M)$ is a projective cover in $\mod A$. \label{forget3}
\item $F$ sends indecomposable objects in $\mod^{\Z}A$ to indecomposable objects in $\mod A$. \label{forget4}
\item Two indecomposable objects $X,Y$ in $\mod^{\Z}A$ are isomorphic in $\mod A$ if and only if $X\simeq Y(i)$ in $\mod^{\Z}A$ for some $i$. \label{forget5}
\end{enumerate}
\end{proposition}

\begin{proof}
Statement (a) is \cite[Proposition 3.5]{GG1}, (b) is \cite[discussion post-Lemma 1.2]{GG2}, (c) is \cite[Proposition 1.3]{GG1}, (d) is \cite[Theorem 3.2]{GG1}, and (e) is \cite[Theorem 4.1]{GG1}.
\end{proof}

\subsection{Trivial extension algebras} \label{subsec:trivext}
Recall that the \emph{trivial extension algebra} $T(A)$ of a finite-dimensional algebra $A$, by definition, is the vector space
\[T(A)=A \oplus D(A) \qquad \text{with multiplication}\qquad (a,f)(b,g)=(ab,ag+fb)\]
for $a,b \in A$, $f,g \in D(A)$, where $D(A)$ is viewed as an $A$-$A$-bimodule.
It has a natural grading, given by $T(A)_0=A$ and $T(A)_1=DA$, and whenever we refer to $T(A)$ as a graded algebra, it is this grading that we have in mind.
The \emph{repetitive category} of $A$ is the category $\widehat{A}=\proj^{\Z}T(A)$ of graded projective $T(A)$-modules.
  It can be viewed as an algebra of infinite matrices of the form
   \[\widehat{A}=\begin{pmatrix}
 \ddots & & & & & \\
 \ddots & A & & & \\
  & DA & A & &\\
  & & DA &A & \\
 & & & \ddots & \ddots 
 \end{pmatrix},
   \]
see \cite{Hap} for more details.
For a positive integer $r$, the \emph{$r$-fold trivial extension} of $A$ is the category $T_r(A) = \proj^{\Z/r\Z}T(A)$. As an algebra, it is the orbit algebra $\widehat{A}/\langle\nu_{\widehat{A}}^r\rangle$ (see Section~\ref{section:selfinjective}) where $\nu_{\widehat{A}}$ is the Nakayama automorphism of $\widehat{A}$, and hence isomorphic to the $r\times r$ matrix algebra
\[
T_r(A)=
  \begin{pmatrix} 
A & & & & DA \\
DA & A & & & \\
 & DA & \ddots & &\\
 & & \ddots &\ddots & \\
& & & DA & A 
\end{pmatrix} ,
\]
for $r\ge2$, whilst $T_1(A) = T(A)$.

It is elementary that $\widehat{A}$ and $T_r(A)$ are self-injective, and their Nakayama automorphisms are given by (cyclic) shift one step down and right in the matrix. 
In particular, $T_r(A)$ is symmetric if and only if $r=1$.
Moreover, there are equivalences
\begin{equation}\label{mod Z/nZ B}
\mod\widehat{A}\simeq\mod^{\Z}T(A)\ \mbox{ and }\ \mod T_r(A)\simeq\mod^{\Z/r\Z}T(A).
\end{equation}

A proof of the following lemma can be found in \cite[Lemma 1.9]{FGR}.

\begin{lemma}\label{calculate inverse}
The units of the trivial extension algebra $T(A)$ of $A$ are given by $T(A)^\times = \{ (a,f)\mid a \in A^\times \}$.
The inverse of $(a,f) \in T(A)^\times$ is given by $(a^{-1}, -a^{-1}f a^{-1})$.
\end{lemma}

Later we need the following easy observation.

\begin{lemma} \label{inneriffidentity}
For an element $r \in Z(A)^\times$, let $\varphi_r: T(A) \rightarrow T(A)$ be the $K$-algebra automorphism of $T(A)$ given by $\varphi_r(a,f)=(a,rf)$.
Then $\varphi_r$ is an inner automorphism if and only if $r=1$.
\end{lemma}

\begin{proof}
It suffices to show the ``only if'' part. Assume that $\varphi_r$ is an inner automorphism, given by conjugation with $(a,f)\in T(A)^\times$.
For all $b\in A$, using Lemma~\ref{calculate inverse}, we get
\[(b,0)= \varphi_r(b,0)=(a,f)(b,0)(a,f)^{-1} = (aba^{-1}, fba^{-1} - aba^{-1}fa^{-1})\]
and hence $b=aba^{-1}$, implying that $a\in Z(A)$.
Moreover, for all $g\in DA$, \[(0,rg)=\varphi_r(0,g)=(a,f)(0,g)(a,f)^{-1}=(0,ag)(a^{-1},-a^{-1}fa^{-1})=(0,aga^{-1})=(0,g),\]
and it follows that $r=1$.
\end{proof}

\subsection{Serre functor, fractionally Calabi--Yau algebras and cluster tilting}\label{subsec:CYdim}

The algebra $A$ is said to be \emph{Iwanaga--Gorenstein} if it has finite injective dimension both as a right- and left $A$-module.
For such an algebra $A$, the \emph{Nakayama functor}
\begin{equation}\label{define nu}
\nu:= - \Lotimes_ADA \simeq D\circ \RHom_A(-,A) \:: \per A \to \per A
\end{equation}
is an auto-equivalence of the perfect derived category $\per A=\Kb(\proj A)$ of $A$-modules, satisfying the bifunctorial isomorphism
\[
\Hom_{\per A}(X,Y)\simeq D\Hom_{\per A}(Y,\nu(X)).
\]
In other words, $\nu$ is a \emph{Serre functor} on $\per A$.
In particular, if $A$ has finite global dimension, then it is Iwanaga--Gorenstein and $\per A\simeq\Db(\mod A)$, so $\nu$ gives a Serre functor on $\Db(\mod A)$.
Moreover, in this case, Happel gave a triangle equivalence \cite{Hap}
\begin{equation}\label{happel}
\Db(\mod A)\simeq\stmod^{\Z} T(A).
\end{equation}
The uniqueness of the Serre functor shows that the following diagram commutes up to isomorphism of functors:
\begin{equation}\label{serre}
\xymatrix{
\Db(\mod A)\ar[r]^{\sim}\ar[d]^\nu&\stmod^{\Z}T(A)\ar[d]^{\Omega\circ(1)}\\
\Db(\mod A)\ar[r]^{\sim}&\stmod^{\Z}T(A)\,.}
\end{equation}

\begin{definition}\label{def:fCY}

Let $\ell$ and $m$ be integers, and $\ell\ne0$. An Iwanaga--Gorenstein algebra $A$ is said to be \emph{twisted $\frac{m}{\ell}$-Calabi--Yau} if there is an isomorphism of functors
\begin{equation}\label{eq:tfCY}
\nu^\ell \simeq [m]\circ \phi^*
\end{equation}
on $\per A$ for some $\phi\in\Aut_k(A)$, which we call the \emph{associated twist}.
If $\phi=\mathrm{id}$ then $A$ is \emph{$\frac{m}{\ell}$-Calabi--Yau}.
The algebra $A$ is \emph{(twisted) fractionally Calabi--Yau} if it is (twisted)  $\frac{m}{\ell}$-Calabi--Yau for some $m$ and $\ell$.
For an $\frac{m}{\ell}$-Calabi--Yau algebra $A$, the rational number $m/\ell$ is uniquely determined by $A$.
We write $\CYdim A=(m,\ell)$ for the smallest $m\in\Z$, $\ell\in\Z_{>0}$ such that $A$ is $\frac{m}{\ell}$-Calabi--Yau.
\end{definition}

We refer to Section~\ref{sec:eg} for examples (known and new) of fractionally Calabi--Yau and twisted fractionally Calabi--Yau algebras.

\begin{remark}
By Proposition~\ref{Out Pic}, a twisted fractionally Calabi--Yau algebra is fractionally Calabi--Yau if and only if the order of the associated twist in the outer automorphism group is finite. It is open whether this is always the case -- this is the content of Question \ref{remove twist 2} in the introduction.
\end{remark}

For general triangulated categories, as in the case of the usual Calabi--Yau property \cite[Section~2.6]{K2}, there is a stronger version of the (twisted) fractional Calabi--Yau property, by which \eqref{eq:tfCY} is required to be an isomorphism of \emph{triangle} functors.
However, in the setting of $\per A$, the two versions coincide.
The following characterisations will be used frequently in the sequel.

\begin{proposition}\label{characterise CY}
Assume that $A$ is Iwanaga--Gorenstein.
\begin{enumerate}[\rm(a)]
\item The following statements are equivalent. \label{tfcy}
\begin{enumerate}[\rm(i)]
\item $A$ is twisted $\frac{m}{\ell}$-Calabi--Yau;
\item $(DA)^{\Lotimes_A\ell}\simeq A[m]$ in $\Db(\mod A)$;
\item $(DA)^{\Lotimes_A\ell}\simeq {}_\phi A_1[m]$ in $\Db(\mod A^{\mathrm{e}})$ for some $\phi\in\Aut_k(A)$;
\item there is an isomorphism of triangle functors $\nu^\ell \simeq [m]\circ \phi^*$ for some $\phi\in\Aut_k(A)$.
\end{enumerate}
\item The following are equivalent. \label{fcy}
\begin{enumerate}[\rm(i)]
\item $A$ is $\frac{m}{\ell}$-Calabi--Yau; 
\item $(DA)^{\Lotimes_A\ell}\simeq A[m]$ in $\Db(\mod A^{\mathrm{e}})$;
\item there is an isomorphism of triangle functors $\nu^\ell \simeq [m]$. \end{enumerate}
\end{enumerate}
\end{proposition}

\begin{proof}
  We only prove \eqref{tfcy} since the proof of \eqref{fcy} is parallel. 
The equivalences (i)$\Leftrightarrow$(ii)$\Leftrightarrow$(iii)
are \cite[Prop 4.3]{HI} and its proof. Note that $\gldim A<\infty$ is assumed there, but Iwanaga--Gorensteiness is enough. 
The implication (iv)$\Rightarrow$(i) holds by definition, and  (iii)$\Rightarrow$(iv) follows from the isomorphism
\[\nu^\ell=-\Lotimes_A(DA)^{\Lotimes_A\ell}\simeq -\Lotimes_A{}_\phi A_1[m]=[m]\circ\phi_* .\qedhere\]
\end{proof}

We will give examples of fractionally Calabi--Yau algebras in Section \ref{example of CY}.


\section{Preliminaries on periodicity and twisted periodicity}\label{section 2}

In this section, we review various notions of periodicity of algebras and modules. Recall that $A$ denotes a finite-dimensional algebra over an arbitrary field $k$.

\begin{definition} \label{def:periodic}
\begin{enumerate}[\rm(a)]
\item An $A$-module $M$ is \emph{$\Omega$-periodic} if there is some integer $n>0$ such that $\Omega^n(M)\simeq M$ in $\mod A$.
\item The algebra $A$ is \emph{(bimodule) periodic} if it is $\Omega$-periodic as a $A^{\mathrm{e}}$-module, i.e.\ $\Omega_{A^{\mathrm{e}}}^n(A)\simeq A$ in $\mod A^{\mathrm{e}}$ for some integer $n>0$. In this case, we call $A$ \emph{$n$-periodic}.

\item $A$ is \emph{twisted (bimodule) periodic} if $\Omega_{A^{\mathrm{e}}}^n(A)\simeq {}_1A_\phi$ in $\mod A^{\mathrm{e}}$ for some integer $n>0$ and $\phi\in \Aut_k(A)$.  We call $\phi$ the \emph{associated twist}.
\end{enumerate}
\end{definition}

Clearly, periodic algebras are twisted periodic, but it is still open whether the converse holds -- this is the content of Question \ref{conj:periodicity} in the introduction. 

\begin{remark}\label{periodic conjecture via twist}
By Proposition~\ref{Out Pic}, a twisted periodic algebra is periodic if and only if the order of the associated twist in the outer automorphism group is finite.
\end{remark}

We start by listing a few observations that will be useful for us later.
The property (d) below was pointed out to us by \O yvind Solberg.
\begin{proposition}\label{elementary properties}
\begin{enumerate}[\rm(a)]
\item $A\times B$ is periodic if and only if both $A$ and $B$ are periodic.
\item Periodic algebras are self-injective.
\item Periodicity is preserved by derived (and hence, Morita) equivalence.
\item If $k$ is a field of characteristic different from two, then the period of any periodic finite-dimensional $k$-algebra is even.
\end{enumerate}
\end{proposition}
\begin{proof}
(a)
As an $(A\times B)^{\mathrm{e}}$-module, $A\times B\cong A\oplus B$ with the obvious action. Hence, $\Omega_{(A\times B)^{\mathrm{e}}}^n(A\times B)\cong \Omega_{A^{\mathrm{e}}}^n(A)\oplus \Omega_{B^{\mathrm{e}}}^n(B)$ for all $n$, and the equivalence follows.

(b) This is \cite[1.4]{GSS} without the ring-indecomposability condition, which is superfluous by (a).

(c) This is \cite[Thm 2.9]{ES} with conditions relaxed thanks to (a) and (b).

(d) Let $A$ be $p$-periodic, so that $\Omega^p_{A^{\mathrm{e}}}(A)\simeq A$. 
Then the element $x\in\mathrm{HH}^p(A)=\Ext_{A^{\mathrm{e}}}^p(A,A)$ in the $p$-th Hochschild cohomology corresponding to the first part of the minimal projective resolution
\[
0\to A \rightarrow P_{p-1} \rightarrow \cdots \rightarrow P_0 \rightarrow A\to 0
\]
of $A$ as an $A^{\mathrm{e}}$-module is not nilpotent; see, e.g.\ \cite[1.3]{GSS}.
On the other hand, since the Hochschild cohomology ring is graded-commutative \cite[Corollary~1]{Ger},
and $\mathop{\rm char} k \ne2$,
 every element $y$ in odd degree satisfies $y^2=0$. Thus $p$ must be even.
\end{proof}

\begin{example}\label{example of periodic algebra}
Two fundamental classes of examples of bimodule periodic algebras are preprojective algebras \cite{ES, Bu}
 (see also \cite{AR}) and trivial extension algebras of Dynkin type \cite{BBK}.
Further examples include:

\begin{enumerate}[\rm(a)]
\item self-injective algebras of finite representation type over algebraically closed fields \cite{Du2};
\item mesh algebras of Dynkin type \cite[Section~6]{BBK}, \cite{Du4};
\item weighted surface algebras (with four exceptions) \cite{ES6} and a closely related family called algebras of generalized quaternion type \cite{ES4}, over algebraically closed fields;
\item blocks of finite group algebras over algebraically closed fields with positive characteristic such that the defect group is cyclic (by (a)) or generalized quaternion (by (c)).
\end{enumerate}
\end{example}

We start our treatment by recalling some equivalent conditions for twisted periodicity.

\begin{proposition}\cite{GSS}\label{thm:hani}
Assume that $A/\rad A$ is a separable $k$-algebra.
Then, for any $n\in \Z_{>0}$, the following statements are equivalent.
\begin{enumerate}[\rm(i)]
\item {\rm (Simple periodicity)} $\Omega^n(A/\rad A)\simeq A/\rad A$ in $\stmod A$. \label{thm:hani1}
\item {\rm (Twisted functorial periodicity)}
  There exists some automorphism $\phi\in \Aut_k(A)$ such that $\Omega^n\simeq \phi^*$ as autoequivalences of $\stmod A$.
\item {\rm (Twisted periodicity)} $\Omega_{A^{\mathrm{e}}}^n(A)\simeq {}_{1}A_{\psi}$ in $\stmod A^{\mathrm{e}}$ for some automorphism $\psi\in \Aut_k(A$).
\end{enumerate}
\end{proposition}
Note that simple periodicity, condition~\eqref{thm:hani1}, holds if and only if all simple $A$-modules are periodic (although not necessarily of the same period $n$).

\begin{proof}
This is \cite[1.4]{GSS} with the following conditions removed:
Firstly, the assumption that $k$ is algebraically closed can be replaced by the separability of the $k$-algebra $A/\rad A$; see \cite[2.1]{H}. Secondly, ring-indecomposable is dropped by Proposition \ref{elementary properties} (a).  Lastly, the idempotent-fixing property in \cite[1.4(b)]{GSS} (which corresponds to (iii) here) is automatic if we replace $\psi$ by a suitable power $\psi^m$.
\end{proof}

All notions of periodicity admit a graded analogue for graded algebras. In the case of $T(A)$, these notions provide a middle ground which serves to make translations between (ungraded) periodicity properties of $T(A)$ and fractional Calabi--Yau properties of $A$.

\begin{definition} \label{def:gr-periodic}
Assume that $A$ is graded.
\begin{enumerate}[\rm(a)]
\item A graded module $M\in \mod^\Z A$ is \emph{graded $\Omega$-periodic} if there exist some integers $n>0$ and $a\in\Z$ such that $\Omega^n(M)\simeq M(a)$ in $\mod^{\Z}A$.

\item The algebra $A$ is \emph{graded (bimodule) periodic} if it is graded $\Omega$-periodic as a graded $A^{\mathrm{e}}$-module.
\item The algebra $A$ is \emph{graded twisted (bimodule) periodic} if $\Omega_{A^{\mathrm{e}}}^n(A)\simeq {}_1A_\phi(a)$ in $\mod A^{\mathrm{e}}$ for some integers $n>0$, $a\in\Z$ and a graded automorphism $\phi\in \Aut_k^\Z(A)$.
\end{enumerate}
\end{definition}

We have the following equivalent conditions of graded twisted periodicity, similar to Proposition~\ref{thm:hani}.

\begin{proposition}{\cite[2.4]{H}}\label{thm:hani-gr}
Assume that $A$ is graded, that $A/\rad A$ is a separable $k$-algebra, and let $a\in\Z$, $n\in \Z_{>0}$.
The following statements are equivalent.
\begin{enumerate}[\rm(i)]
\item {\rm (Graded simple periodicity)} $\Omega^n(A/\rad A)\simeq (A/\rad A)(a)$ in $\stmod^{\Z}A$.

\item {\rm (Graded twisted functorial periodicity)}
  There exists a graded automorphism $\phi\in \Aut_k^\Z(A)$ of $A$ such that $\Omega^n\simeq \phi^*\circ(a)$ as autoequivalences of $\stmod^{\Z} A$.

\item {\rm (Graded twisted periodicity)} $\Omega_{A^{\mathrm{e}}}^n(A)\simeq {}_{1}A_{\psi}(a)$ in $\stmod^{\Z}A^{\mathrm{e}}$ for some graded automorphism $\psi\in \Aut_k^\Z(A)$.
\end{enumerate}
\end{proposition}

\begin{proof}
This is \cite[2.4]{H}. The assumption in \cite{H} that $A$ is ring-indecomposable and non-semisimple is superfluous, by Proposition~\ref{elementary properties}.
\end{proof}

Beware that, unlike in the ungraded case, graded periodicity of all simple modules does not always imply the conditions in Proposition~ \ref{thm:hani-gr}(i).
Thankfully, it turns out that the two notions coincide when $A$ is ring-indecomposable.

\begin{proposition}\label{gr=ungr}
Assume that $A$ is ring-indecomposable and graded, and that $A/\rad A$ is a separable $k$-algebra. Then the following statements hold.
\begin{enumerate}[\rm(a)]
\item $A$ is twisted periodic if, and only if, $A$ is graded twisted periodic. \label{gr=ungr1}
\item $A$ is periodic if, and only if, $A$ is graded periodic. \label{gr=ungr2}
\end{enumerate}
\end{proposition}

\begin{proof}
(a)
By Proposition~\ref{forgetful}\eqref{forget1}, each simple object in $\mod A$ is gradable.
By Proposition~\ref{forgetful}\eqref{forget3}, the forgetful functor $\mod^{\Z}A\to\mod A$ sends minimal projective resolutions in $\mod^{\Z}A$ to minimal projective resolutions in $\mod A$. Thus, by Proposition~ \ref{forgetful}(\ref{forget4},\ref{forget5}), a simple object in $\mod^{\Z}A$ is graded $\Omega$-periodic if and only if it is $\Omega$-periodic.

Clearly, $A/\rad A$ is $\Omega$-periodic if and only if each simple $A$-module is $\Omega$-periodic.
Further, we claim that $A/\rad A$ is graded $\Omega$-periodic if and only if each simple object in $\mod^{\Z}A$ is graded $\Omega$-periodic.
The ``only if'' part is clear. To prove the ``if'' part, take a common $n>0$ satisfying  $\Omega^n(S)\simeq S(a_S)$ for some $a_S\in\Z$ for each simple object $S\in\mod^{\Z}A$. 
Then, for simple $A$-modules $S$ and $T$, $\Ext^1_A(S,T)\neq0$ implies $a_S=a_T$. In fact, there exists a finite subset $\emptyset\neq I\subset\Z$ such that $\Ext^1_{\mod^{\Z}A}(S,T(i))\neq0$ holds if and only if $i\in I$. Then
\[\Ext^1_{\mod^{\Z}A}(S,T(i))\simeq\Ext^1_{\mod^{\Z}A}(\Omega^n(S),\Omega^n(T(i)))\simeq\Ext^1_{\mod^{\Z}A}(S,T(a_T-a_S+i))\]
implies $I=a_T-a_S+I$ and hence $a_S=a_T$.
Consequently ring-indecomposability of $A$ implies that $a_S=a_T$ for all simples $S,T$. Thus, $\Omega^n(A/\rad A) \simeq(A/\rad A)(a_S)$.

Consequently, $A/\rad A$ is graded $\Omega$-periodic if and only if it is $\Omega$-periodic. The desired equivalence now follows from Propositions~\ref{thm:hani} and \ref{thm:hani-gr}.

(b) Since $A$ is ring-indecomposable, it is indecomposable as an object in $\mod^{\Z}A^{\mathrm{e}}$.  Thus the assertion follows from Proposition~\ref{forgetful}(\ref{forget3},\ref{forget4},\ref{forget5}) when applying the forgetful functor $\mod^{\Z}A^{\mathrm{e}}\to\mod A^{\mathrm{e}}$ to the minimal projective resolution of $A$ in $\mod^{\Z}A^{\mathrm{e}}$.
\end{proof}

We recall the notion of complexity of a module.

\begin{definition}
Let $M$ be a $A$-module, and $(P_n)_{n\geq 0}$ a minimal projective resolution of $M$. The \emph{complexity of $M$} is defined as
\[
\cx_A(M) = \inf\{ n\in\mathbb{N}\cup\{\infty\} \mid \exists\, C\in\mathbb{N}\, \forall\,t\in\mathbb{N}: \dim_k P_t\leq Ct^{n-1}\}.
\]
\end{definition}

\begin{remark}\label{remark on cx}
Note that $\cx_A(M) = 0$ is equivalent to $\projdim M < \infty$, and that $\cx_A(M)\leq 1$ holds if and only if there is a bound on the dimensions of the terms $P_t$ in a minimal projective resolution of $M$.
In particular, any $\Omega$-periodic module has complexity one.
  Over a twisted periodic algebra, all simple modules are $\Omega$-periodic (Proposition~\ref{thm:hani}) and thus of complexity one, whence $\cx_A(M)\le1$ holds for all $M \in \mod A$ by the Horseshoe lemma.
\end{remark}

\section{Characterisations of twisted periodicity}\label{sec:periodic main}

The overall goal in this paper is to establish the implications in Figure~\ref{dgm:implications}.
The first column of this figure shows properties of an algebra $A$ of finite global dimension, whereas the second and third columns concern properties of its trivial extension algebra $T(A)$. For completeness, the (known) implications given by Proposition~\ref{thm:hani}, Proposition~\ref{thm:hani-gr} and Remark~\ref{remark on cx} are also included in the diagram.

\begin{figure}[!htbp]
\centering
\begin{tikzpicture}[every text node part/.style={align=center}] 
\tikzset{
    iff/.style={implies-implies,double equal sign distance},
    imp/.style={-implies,double equal sign distance},
  }
  
\node (v3) at (0,0) {fractionally \\ Calabi--Yau};
\node (v1) at (4.5,0) {graded \\ periodic};
\node (v2) at (9,0) {periodic};
\draw [iff] (v1) -- node [above] {\footnotesize Prop \ref{gr=ungr}\eqref{gr=ungr2}} (v2);
\draw [imp] (1.1,0.2) -- node [above] {\footnotesize Thm \ref{trivial extension is periodic}}(3.7,0.2);
\draw [imp] (3.7,-0.2) -- node [below] {\footnotesize Prop \ref{periodic to fCY}}(1.1,-0.2);

\node (v4) at (0,-3) {twisted \\ fractionally \\ Calabi--Yau};
\node (v5) at (4.5,-3) {graded \\ twisted \\ periodic};
\node (v6) at (9,-3) {twisted \\ periodic};
\draw [imp] (v4) -- node[above]{\footnotesize Prop \ref{tw-fCY to tw-pe}} (v5);
\draw [iff] (v5) -- node[above] {\footnotesize Prop \ref{gr=ungr}\eqref{gr=ungr1}} (v6);

\draw [imp] (v3) -- node [fill=white] {\footnotesize trivial} (v4); 
\draw [imp] (v1) -- node [fill=white] {\footnotesize trivial} (v5);
\draw [imp] (v2) -- node [fill=white] {\footnotesize trivial} (v6); 

\node (v7) at (4.5,-6) {graded \\ twisted \\ functorially \\ periodic};
\node (v8) at (4.5,-9) {graded \\ simple \\ periodic};
\draw [iff] (v5) -- node [left] {\footnotesize Prop \ref{thm:hani-gr}} (v7);
\draw [iff] (v7) -- node [left] {\footnotesize Prop \ref{thm:hani-gr}} (v8);

\node (v10) at (9,-6) {twisted \\ functorially  \\ periodic};
\node (v9) at (9,-9) {simple \\ periodic};
\draw [iff] (v6) -- node [right, align=left] {\footnotesize Prop \ref{thm:hani}} (v10);
\draw [iff] (v10) -- node [right, align=left] {\footnotesize Prop \ref{thm:hani}} (v9);
\draw [iff] (v8) -- node[above]{\scriptsize Proof of \\ \footnotesize Prop \ref{gr=ungr}\eqref{gr=ungr1}} (v9);

\node (v11) at (8.5,-11) {$\cx_{T(A)}({}^{\forall}M)\leq 1$};
\draw [imp] (v9) -- node[right,align=left]{\footnotesize Rem \ref{remark on cx}} +(0,-1.5);
\draw [imp] (7,-11) .. controls +(-7.5,0) and (0,-11.5) .. (0,-3.7);
\node at (1,-9.5) {\footnotesize Prop \ref{complexity one}};

\node at (0,1.5) {Properties of $A$ \\ with $\mathrm{gldim} A<\infty$};
\node at (6,1.5) { \phantom{abc}\\ Properties of the trivial extension $T(A)$};
\draw [decorate,decoration={brace,amplitude=7pt}] (-1,.7) -- (1,.7);
\draw [decorate,decoration={brace,amplitude=7pt}] (2.5,0.7) -- (10,0.7);
\end{tikzpicture}
\caption{Relations between various notions of periodicity}\label{dgm:implications}
\end{figure}
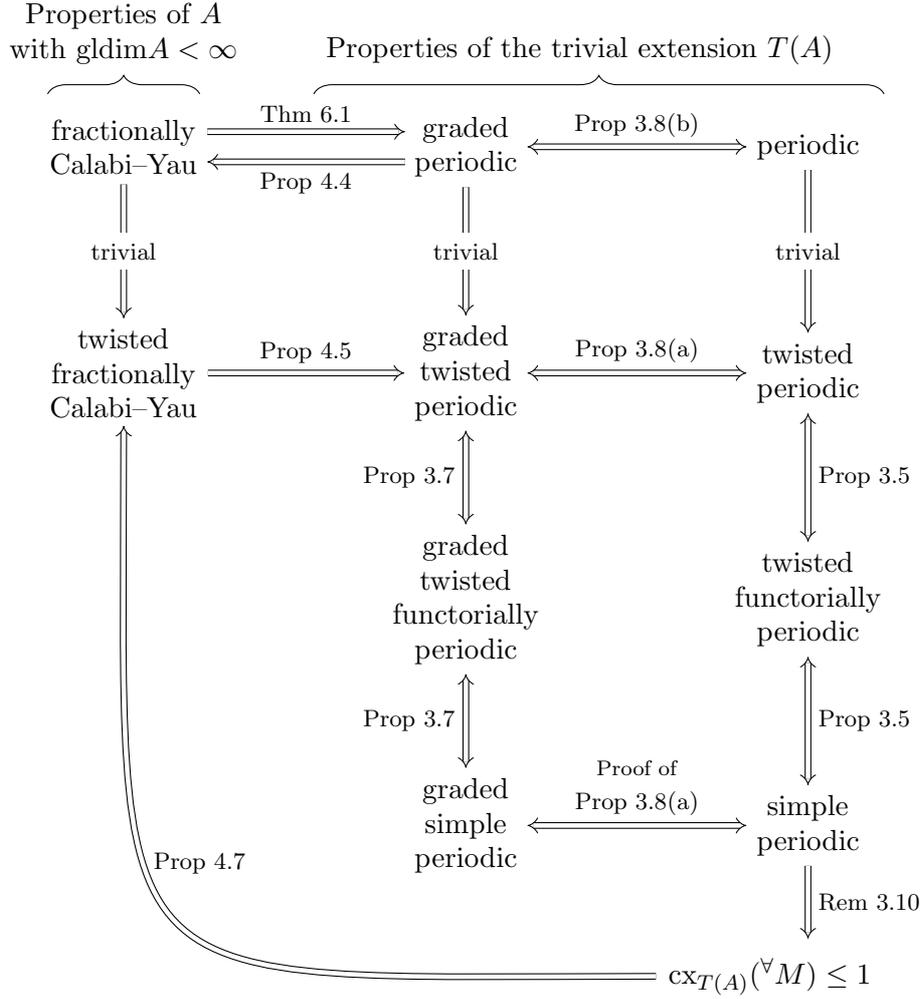

Notably, if Question~\ref{conj:periodicity} has an affirmative answer (i.e., the implication arrow $\Downarrow$ in the upper-right corner of Figure~\ref{dgm:implications} can be upgraded to an equivalence), then all the conditions in the diagram are equivalent. In particular, every \emph{twisted} fractionally Calabi--Yau  algebra will necessarily be (untwisted) fractionally Calabi--Yau -- albeit of different minimal dimension in most cases. This would resolve a question in \cite[Remark 1.6(b)]{HI}.

In this section, we shall prove some implications between the twisted fractional Calabi--Yau property of $A$ and different notions of twisted periodicity for $T(A)$. These results, summarized in Theorem~\ref{main theorem 2 again} below, imply that all the ``twisted'' notions in our setting are, in fact, equivalent. We also give one result, Proposition~\ref{periodic to fCY}, about properties without twist: if $T(A)$ is graded periodic then $A$ is fractionally Calabi--Yau.

We need the following key notion before stating~Theorem \ref{main theorem 2 again}.

\begin{definition}\label{def:CT}
Let $A$ be a finite dimensional algebra over a field $k$, and $d$ a positive integer. An $A$-module $M$ is said to be \emph{$d$-cluster-tilting} if
\begin{align*}
  \add M
  &= \{ X \in \mod A \mid \Ext_{A}^i(X,M)=0, \ \forall\, 0 < i < d \}\\
  &= \{ X \in \mod A \mid \Ext_{A}^i(M,X)=0, \ \forall\, 0 < i < d \}.
\end{align*}
The algebra $A$ is said to be \emph{$d$-representation-finite}\footnote{Notice that it was called \emph{weakly $d$-representation-finite} in \cite{HI,IO} and elsewhere.} if it has a $d$-cluster-tilting module $M$. In this case $\End_A(M)$ (respectively, $\underline{\End}_A(M)$) is a \emph{$d$-Auslander algebra} (respectively, \emph{stable $d$-Auslander algebra}) of $A$.
\end{definition}

We are ready to state the main result of this section.

\begin{theorem}\label{main theorem 2 again}
Let $A$ be a finite-dimensional algebra over a field $k$ such that $A/\rad A$ is a separable $k$-algebra. The following conditions are equivalent.
\begin{enumerate}[\rm(i)]
\item $T(A)$ is twisted periodic.
\item $T(A)$ is graded twisted periodic.
\item Every $T(A)$-module has complexity at most one.
\item There exist integers $d,r\ge1$ such that $T_r(A)$ is $d$-representation-finite.
\item $A$ has finite global dimension and is twisted fractionally Calabi--Yau.
\end{enumerate}
\end{theorem}

We remark that the implication (iii) $\Rightarrow$ (i) in Theorem~\ref{main theorem 2 again} has been proved independently by Dugas in the preprint \cite{Du1}, for an arbitrary self-injective algebra $B$ over an algebraically closed field. 

As a consequence of Theorem~\ref{main theorem 2 again}, we can answer a question in \cite[Remark 1.6(c)]{HI}.

\begin{corollary}\label{derived closed}
Let $A$ and $B$ be derived equivalent finite-dimensional algebras of finite global dimension such that $A/\rad A$ and $B/\rad B$ are separable $k$-algebras.
Then $A$ is twisted fractionally Calabi--Yau if and only if so is $B$.
\end{corollary}

\begin{proof}
Suppose $A$ is twisted fractionally Calabi--Yau.  By Theorem \ref{main theorem 2 again}, we have $\cx_{T(A)}(M)\leq 1$ for every $T(A)$-module $M$.

By Rickard's results \cite[3.1, 2.2]{Ric}, we have a stable equivalence $F:\stmod T(B)\xrightarrow{\sim} \stmod T(A)$.  It then follows from \cite[Theorem 4.5]{Pu}, which says the complexity of a module is preserved under stable equivalence, that $\cx_{T(B)}(N)=\cx_{T(A)}(F(N))\leq 1$ for any  $T(B)$-module $N$.  By Theorem \ref{main theorem 2 again}, this means that $B$ is also twisted fractionally Calabi--Yau.
\end{proof}

Happel's triangle equivalence \eqref{happel} provides a way to translate periodicity properties of $T(A)$ into properties of the derived category $\Db(\mod A)$ of $A$. As can be seen from the commutative diagram \eqref{serre}, the grading shift functor $(1):\stmod^{\Z} T(A)\to \stmod^{\Z} T(A)$ corresponds under \eqref{happel} to the autoequivalence $\nu\circ[1]$ of $\Db(\mod A)$. This immediately leads to the following observation.

\begin{proposition}\label{periodic to fCY}
Assume that $A$ has finite global dimension. Given $m,\ell\in\Z$ with $\ell>0$, the following are equivalent.
\begin{enumerate}[\rm(a)]
\item
  The algebra $A$ is $\frac{m}{\ell}$-Calabi--Yau.
\item
There is an isomorphism of functors $\Omega^{\ell+m}\simeq  (-\ell)$ on $\stmod^{\Z}T(A)$.
\end{enumerate}
In particular, if $T(A)$ is (graded) periodic, then $A$ is fractionally Calabi--Yau.
\end{proposition}

\begin{proof}
  By \eqref{serre}, we have $\nu^\ell\circ[-m]\simeq \Omega^{\ell+m}\circ(\ell)$. Thus (a) is equivalent to (b). Moreover, graded periodicity of $T(A)$ implies the condition (b) (cf.\ Proposition \ref{thm:hani-gr}(iii)$\Rightarrow$(ii)).
  In fact, any isomorphism $\Omega_{T(A)^{\mathrm{e}}}^{n}(T(A))\simeq T(A)(a)$ in $\mod^{\Z}(T(A)^{\mathrm{e}})$ gives natural isomorphisms 
  \[\Omega_{T(A)}^{n} \simeq  -\otimes_{T(A)}\Omega_{T(A)^{\mathrm{e}}}^{n}(T(A)) \simeq -\otimes_{T(A)}T(A)(a) = (a)\ \mbox{ on }\ \stmod^\Z T(A).\qedhere\]
\end{proof}

The following twisted version of Proposition~\ref{periodic to fCY} gives the implication (v)$\Rightarrow$(ii) in Theorem~\ref{main theorem 2 again}.

\begin{proposition}\label{tw-fCY to tw-pe}
Assume that $A$ has finite global dimension.
If $A$ is twisted $\frac{m}{\ell}$-Calabi--Yau, then there exists an automorphism $\phi\in\Aut_k^{\Z}(T(A))$ such that
\[\Omega^{m+\ell}\simeq\phi^*\circ(-\ell)\ \mbox{ on }\ \stmod^\Z T(A).\]
\end{proposition}

\begin{proof}
By \eqref{serre}, we have $\psi^*\simeq\nu^\ell\circ[-m]\simeq \Omega^{\ell+m}\circ(\ell)$ for some automorphism $\psi\in \Aut_k(A)$ of $A$.
The simple $A$-modules correspond under \eqref{happel} to the simple $T(A)$-modules concentrated in degree $0$.
Let $X:=A/\rad A=T(A)/\rad T(A)$. Then $\psi^*(X)\simeq X$ holds, and
the above implies that $\Omega^{\ell+m}(X) \simeq X(-\ell)$ in $\stmod^{\Z} T(A)$.
The claim now follows from Proposition \ref{thm:hani-gr}.
\end{proof}

Next we observe the following simple fact.

\begin{proposition}\label{fin gl}
  If $T(A)$ is (graded) twisted periodic, then $\gldim A$ is finite.
\end{proposition}

\begin{proof}
Let $S$ be a simple $A$-module. By Proposition~\ref{thm:hani-gr}, there exist $m\in\Z_{>0}$ and $a\in\Z$ such that $\Omega_{T(A)}^m(S)\simeq S(a)$. Clearly, $\Omega_{T(A)}^i(S)_{>0}\ne0$ for all $i>0$, implying that $a<0$ and, consequently,
$\Omega_A^m(S) = \Omega_{T(A)}^m(S)_0 = 0$.
Thus $A$ has finite global dimension.
\end{proof}

The following result closes the circuit of implications in the lower part of Figure~\ref{dgm:implications}, thus establishing the equivalence between the twisted fractional Calabi--Yau property of $A$ and twisted periodicity of $T(A)$.

\begin{proposition}\label{complexity one}
Assume that $A$ is ring-indecomposable. If each $T(A)$-module has complexity at most one, then $A$ has finite global dimension and is twisted fractionally Calabi--Yau.
\end{proposition}

For the proof of Proposition~\ref{complexity one}, we need the following well-known result.

\begin{proposition}[{e.g.\ \cite[Corollary 9]{HS}}]\label{voigt}
For each $n\ge1$, there are only finitely many isomorphism classes of $A$-modules $X$ satisfying $\dim_kX=n$ and $\Ext^1_A(X,X)=0$.
\end{proposition}

\begin{proof}
When $k$ is algebraically closed, the assertion follows from Voigt's Lemma \cite{V}. For the general case, let $\overline{k}$ be the algebraic closure of $k$ and $\overline{A}:=\overline{k}\otimes_kA$. The functor $\overline{(-)}:\mod A\to\mod\overline{A}$ satisfies $\dim_{\overline{k}}\overline{X}=\dim_kX$ and $\Ext^1_{\overline{A}}(\overline{X},\overline{X})=\overline{k}\otimes_k\Ext^1_A(X,X)$.
Moreover, $X\simeq Y$ holds as $A$-modules if and only if $\overline{X}\simeq\overline{Y}$ as $\overline{A}$-modules. Thus the assertion follows.
\end{proof}

We also need the following partial answer to the finitistic dimension conjecture.

\begin{proposition}[{\cite[12.63]{JL}\cite[2.2]{MS}}]\label{FDC}
For all $n,n'\ge1$, there exists an integer $m\ge1$ satisfying the following:
if $A$ is a $k$-algebra over a field $k$ with $\dim_kA\le n$, and $X\in\mod A$ with $\dim_kX\le n'$, then the projective dimension $\projdim_AX$ is either $\infty$ or at most $m$.
\end{proposition} 

Now we are ready to prove Proposition~\ref{complexity one}.

\begin{proof}[Proof of Proposition \ref{complexity one}]
Assume that each $T(A)$-module has complexity at most one.

(1) We first show that $A$ has finite global dimension. Suppose, to the contrary, that $\gldim A = \infty$. Then there exists a simple $A$-module $S$ such that $\projdim_AS=\infty$. Take a minimal projective resolution
\[\cdots\to P_2\xrightarrow{f_2} P_1\xrightarrow{f_1} P_0\to S\to0\]
of $S$ in $\mod^{\Z}T(A)$.  By the complexity assumption, there exists some $N>0$ such that $\dim_k\Omega_{T(A)}^i(S)\le N$ for all $i\ge0$.

Since $\Omega_{T(A)}^i(S)_0=\Omega_A^i(S)\neq0$ and $\Omega_{T(A)}^i(S)$ is indecomposable, the inequality $\dim_k\Omega_{T(A)}^i(S)\le N$ implies that $\Omega_{T(A)}^i(S)_N=0$. Let $n\in\Z_{+}$ be  minimal with the property that $\Omega_{T(A)}^i(S)_{n+1}=0$ for all $i\ge0$. Then each $P_i$ is generated by elements of degree at most $n-1$, and the sequence
\begin{equation}\label{degree n part}
\cdots\to (P_2)_n\xrightarrow{(f_2)_n} (P_1)_n\xrightarrow{(f_1)_n} (P_0)_n\to S_n=0
\end{equation}
is exact. In particular, $(P_i)_n\in\inj A$ holds for each $i\ge0$.

We claim that $(f_i)_n$ is a radical map for all $i\ge1$.  Indeed, for each $i$, let $Q_i$ be a maximal direct summand of $P_i$ such that $Q_i$ is generated in degree $n-1$, and let $g_i:Q_i\to Q_{i-1}$ be the composition $Q_i\subset P_i\xrightarrow{f_i}P_{i-1}\to Q_{i-1}$, which is a radical map as $f_i$ by assumption is radical. Then the fact that $P_i$'s being generated by elements of degree at most $n-1$ means that
\[(f_i)_n=(g_i)_n:(P_i)_n=(Q_i)_n\to (P_{i-1})_n=(Q_{i-1})_n\]
holds. Moreover, the morphism $(g_i)_n$ in $\inj A$ is the image of the morphism $(g_i)_{n-1}$ in $\proj A$ under the Nakayama functor. Since $g_i$ is a radical map, so is $(g_i)_{n-1}$, and thus $(f_i)_n$, as desired.

Let $m$ be the minimal number such that $\Omega_{T(A)}^m(S)_n\neq0$. Then the exact sequence \eqref{degree n part} shows that $\injdim_A\Omega_{T(A)}^{m+i}(S)_n=i$ for each $i\ge0$. Since $\dim_k\Omega_{T(A)}^{m+i}(S)_n\le N$, this contradicts to Proposition \ref{FDC}.

(2) Now we prove that $A$ is twisted fractionally Calabi--Yau.
Since the kernel of the canonical projection $T(A)\to A,\:(a,f)\mapsto a$ is concentrated in degree $1$, we obtain $\Ext^1_{\mod^{\Z/2\Z}T(A)}(A,A)=0$. Setting $C=T_2(A)$, and identifying $\mod^{\Z/2\Z}T(A)$ with $\mod C$ by \eqref{mod Z/nZ B}, we obtain $\Ext^1_{C}(\Omega_{C}^i(A),\Omega^i_{C}(A))=\Ext^1_{C}(A,A)=0$ for each $i\ge0$.
Since the complexity assumption says that $\dim_k\Omega_{C}^i(A)=\dim_k\Omega_{T(A)}^i(A)$ is bounded, Proposition \ref{voigt} shows that there exists $n\ge1$ such that $\Omega_{C}^n(A)\simeq A$ in $\mod C$, or equivalently, $\Omega_{T(A)}^n(A)\simeq A$ in $\mod^{\Z/2\Z}T(A)$.
Hence, there exists some multiple $s$ of $n$ such that $\Omega_{T(A)}^{s}(P)\simeq P$ holds in $\mod^{\Z/2\Z}T(A)$ for each indecomposable projective $A$-module $P$.
Consequently, there exists $\ell_P\in 2\Z_{>0}$ such that $\Omega_{T(A)}^s(P)\simeq P(-\ell_P)$ in $\mod^{\Z}T(A)$.
Since $A$ is ring-indecomposable, $\ell=\ell_P$ is independent of $P$. Thus $\Omega_{T(A)}^s(A)\simeq A(-\ell)$ holds in $\mod^{\Z}T(A)$. Using the commutative diagram \eqref{serre}, we have $A[-s]\simeq\nu^{-\ell}(A)[-\ell]$ in
$\Db(\mod A)$. Thus $A$ is twisted fractionally Calabi--Yau.
\end{proof}

The following result from \cite{DI} gives the implication (v)$\Rightarrow$(iv) in Theorem~\ref{main theorem 2 again}. Note that while in \cite[Cor~2.9]{DI}, the algebra is assumed to be basic, this is necessary only to ensure that the given $d$-cluster-tilting module is basic (which is inconsequential for our purposes).

\begin{proposition}\cite[Cor~2.9]{DI} \label{DI2.9}
Assume that $A$ is twisted $\frac{m}{\ell}$-Calabi--Yau, with $\gldim A\le d$ for some positive integer $d$.
Set $g={\rm gcd}(\ell+m,d+1)$ and $r=\frac{(d+1)\ell-(\ell+m)}{g}=\frac{d\ell-m}{g}$. 
Then
\[ T_r(A)\oplus \bigoplus_{i=1}^r \Omega_{T_r(A)}^{(d+1)i}(A(i)) \in \mod T_r(A)\]
is a $d$-cluster-tilting $T_r(A)$-module.
\end{proposition}

We now have all the pieces needed to put together the proof of Theorem~\ref{main theorem 2 again}.

\begin{proof}[Proof of Theorem \ref{main theorem 2 again}]
(i)$\Leftrightarrow$(ii) was shown in Proposition \ref{thm:hani-gr}, (i)$\Rightarrow$(iii) is immediate, (iii)$\Rightarrow$(v) was shown in Proposition \ref{complexity one}, and (v)$\Rightarrow$(ii) was shown in Proposition \ref{tw-fCY to tw-pe}. 

On the other hand, (v)$\Rightarrow$(iv) follows from Proposition \ref{DI2.9}, and (iv)$\Rightarrow$(iii) was shown in \cite[Theorem 1.1]{EH} since the forgetful functor $\mod T_r
(A)\to\mod T(A)$ preserves the complexities of modules. 
Thus all the conditions (i)--(v) are equivalent.
\end{proof}


\section{On $\ell$-th roots of the $m$-th suspension functor $[m]$}\label{section 4}

In this section, we prepare Theorem \ref{trivial extension is periodic: second part} about a complex $P$ of $A^{\mathrm{e}}$-modules whose $\ell$-th tensor power is isomorphic to $A[m]$. This is necessary to prove Theorem \ref{main theorem 2} in the next section.

\subsection{Reminder on chain complexes}
Since the result in the next subsection requires delicate calculations of complexes especially on signs, we recall here some details and conventions.

For a differential graded algebra $B$, $\CCC(B)$ denotes the category of dg $B$-modules, and $\DDD(B)$ its derived category. In particular, if $B$ is a $k$-algebra concentrated in degree zero, then $\CCC(B)$ coincides with the category of chain complexes in, and $\DDD(B)$ with the derived category of, the abelian category $\Mod B$.

Let $X,Y,Z$ etc{.} be objects in $\CCC(k)$.
The degree of a homogeneous element $x\in X$ is denoted by $|x|=|x|_X$. The tensor product $X\otimes_kY\in\CCC(k)$ is given by
\[(X\otimes_kY)^i=\bigoplus_{j\in\Z}(X^j\otimes_kY^{i-j})\ \mbox{ and }\ d_{X\otimes_kY}(x\otimes y)=d_X(x)\otimes y+(-1)^{|x|}x\otimes d_Y(y)\]
for each homogeneous elements $x\in X$ and $y\in Y$. For morphisms $f\in\Hom_{\CCC(k)}(X,X')$ and $g\in\Hom_{\CCC(k)}(Y,Y')$, a morphism $f\otimes g\in\Hom_{\CCC(k)}(X\otimes_kY,X'\otimes_kY')$ is given by
\[(f\otimes g)(x\otimes y)=f(x)\otimes g(y).\]
For morphisms $f'\in\Hom_{\CCC(k)}(X',X'')$ and $g'\in\Hom_{\CCC(k)}(Y',Y'')$, we clearly have
\begin{equation}\label{f'g'fg}
(f'\otimes g')\circ(f\otimes g)=(f'f)\otimes (g'g).
\end{equation}
We view the canonical isomorphism
\[(X\otimes _kY)\otimes_kZ\simeq X\otimes_k(Y\otimes_kZ)\ \mbox{ in $\CCC(k)$ given by }\ (x\otimes y)\otimes z\mapsto x\otimes(y\otimes z),\]
as an identification, and simply write $X\otimes_kY\otimes_kZ$ for this object.  Now let $n,m\in\Z$. The object $X[n]\in\CCC(k)$ is given by
\[(X[n])^i=X^{i+n}\ \mbox{ and }\ d_{X[n]}(x)=(-1)^nd_X(x).\]
There is a canonical isomorphism
\begin{equation}\label{[m] and [n]}
(X[m])\otimes_k(Y[n])\simeq (X\otimes_kY)[m+n]\ \mbox{ in $\CCC(k)$ given by }\ x\otimes y\mapsto (-1)^{|x|n}x\otimes y
\end{equation}
where $|x|=|x|_{X[m]}$, see \cite[4.1.14]{Ya}. On the other hand, we have a canonical isomorphism
\begin{equation}\label{koszul sign}
X\otimes_kY\simeq Y\otimes_kX\ \mbox{ in $\CCC(k)$ given by }\ x\otimes y\mapsto (-1)^{|x||y|}y\otimes x
\end{equation}
for each homogeneous elements $x\in X$ and $y\in Y$.

Recall that $A$ denotes a finite-dimensional $k$-algebra.
Let $X,Y,Z$ etc{.} be objects in $\CCC(A^{\mathrm{e}})$. Then $X\otimes_AY\in\CCC(A^{\mathrm{e}})$ is the quotient of $X\otimes_kY$ by the subcomplex generated by elements $(x(a\otimes 1))\otimes y-x\otimes(y(1\otimes a))$ for $x\in X$, $y\in Y$ and $a\in A$, with differential induced by $d_{X\otimes_kY}$. 
In particular, with the exception of \eqref{koszul sign}, the sign rules listed above still hold if we replace $\otimes_k$ and $\CCC(k)$ by $\otimes_A$ and $\CCC(A^{\mathrm{e}})$ respectively.  In what follows, we shall frequently use the following canonical isomorphisms, which are special cases of \eqref{[m] and [n]}:
\begin{align*}
l_{X,n}:(A[n])\otimes_AX\simeq X[n]&\ \mbox{ in }\ \CCC(A^{\mathrm{e}})\ \mbox{ given by }1\otimes x\mapsto x,\\
r_{X,n}:X\otimes_A(A[n])\simeq X[n]&\ \mbox{ in }\ \CCC(A^{\mathrm{e}})\ \mbox{ given by }\ x\otimes1\mapsto(-1)^{|x|n}x.
\end{align*}
For the case $n=0$, no signs appear in the isomorphisms $A\otimes_AX\simeq X\simeq X\otimes_AA$. Thus we can safely identify  $A\otimes_AX$ and $X\otimes_AA$ with $X$.

For each $a\in Z(A)$, there are
\begin{equation}\label{left/right multiplication}
\lambda_{X,a}:X\to X\ \mbox{ and }\ \rho_{X,a}:X\to X\ \mbox{ in }\ \CCC(A^{\mathrm{e}})\ \mbox{ given by }\ x\mapsto ax\ \mbox{ and }\ x\mapsto xa,
\end{equation}
respectively. We get
\begin{align}\notag
&\lambda_{X,a}\otimes 1_Y=\lambda_{X\otimes_AY,a},\ 1_X\otimes\rho_{Y,a}=\rho_{X\otimes_AY,a}\ \mbox{ and }\\ \label{1a=a1}
&1_X\otimes \lambda_{Y,a}=\rho_{X,a}\otimes 1_Y:X\otimes_AY\to X\otimes_AY\ \mbox{ in }\ \CCC(A^{\mathrm{e}}).
\end{align}
Moreover, each morphism $f\in\Hom_{\CCC(A^{\mathrm{e}})}(X,Y)$ satisfies
\begin{equation}\label{fa=af}
f\circ \lambda_{X,a}=\lambda_{Y,a}\circ f\ \mbox{ and }\ f\circ \rho_{X,a}=\rho_{Y,a}\circ f.
\end{equation}
If $X\in\CCC(A^{\mathrm{e}})$ and $a\in Z(A)$ satisfy $\lambda_{X,a}=\rho_{X,a}$ as morphisms in $\DDD(A^{\mathrm{e}})$, then we write
\begin{align}\label{a1}
a1_X:=\lambda_{X,a}=\rho_{X,a}\in\End_{\DDD(A^{\mathrm{e}})}(X).
\end{align}

\subsection{Cofibrant root of $A[m]$ and a certain central element}\label{section: z}

In this section, we fix $P\in\CCC(A^{\mathrm{e}})$. For $i\ge0$, we write
\[P^{\otimes i}:=\overbrace{P\otimes_A\cdots\otimes_AP}^i,\]
or simply $P^i$ if there is no danger of confusion. We assume the following three conditions.
\begin{enumerate}
\item[{\rm (R1)}] $P$ is cofibrant (i.e.\ each morphism $f:P\to X$ in $\CCC(A^{\mathrm{e}})$ is null-homotopic if $X$ is acyclic).
\item[{\rm (R2)}] There exist integers $\ell\ge1$ and $m$ and a quasi-isomorphism
\begin{equation}\label{define s}
s:P^{\otimes\ell}\to A[m]\ \mbox{ in }\ \CCC(A^{\mathrm{e}}).
\end{equation}
\item[{\rm (R3)}] For each $a\in Z(A)$, $\lambda_{P,a}=\rho_{P,a}$ holds in $\DDD(A^{\mathrm{e}})$\footnote{$\lambda_{P,a}=\rho_{P,a}$ does not necessarily hold in $\CCC(A^{\mathrm{e}})$.}, see \eqref{left/right multiplication}. Moreover there is an isomorphism
\begin{align*}
Z(A) \simeq \End_{\DDD(A^{\mathrm{e}})}(P)\ \mbox{ given by }\ a\mapsto  a1_P:= \lambda_{P,a}=\rho_{P,a}.
\end{align*}
\end{enumerate}
By (R1), we have functors $P\otimes_A-$ and $-\otimes_AP:\DDD(A^{\mathrm{e}})\to\DDD(A^{\mathrm{e}})$. By (R2), we have isomorphisms
\begin{eqnarray*}\label{define t}
t:=\Big(P\otimes_A(A[m])\xrightarrow{(1_{P}\otimes s)^{-1}}P^{\otimes\ell+1}\xrightarrow{s\otimes 1_{P}}(A[m])\otimes_AP\Big)\ \mbox{ in }\ \DDD(A^{\mathrm{e}}),\\ \label{define u}
t':=\Big(P[m]\xrightarrow{r_{P,m}^{-1}}P\otimes_A(A[m])\xrightarrow{t}(A[m])\otimes_AP\xrightarrow{l_{P,m}}P[m]\Big)\ \mbox{ in }\ \DDD(A^{\mathrm{e}}).
\end{eqnarray*}
By (R3), there is $z\in Z(A)^\times$ such that $t'=z1_{P[m]}$, see \eqref{a1}. These definitions can be summarized by the following commutative diagram in $\DDD(A^{\mathrm{e}})$.
\begin{equation}\label{define tuz}
\xymatrix@R0.5em@C4em{
&P\otimes_A(A[m])\ar[r]^(.6){r_{P,m}}\ar[dd]^t&P[m]\ar[dd]^{t'=z1_{P[m]}}\\
P^{\otimes\ell+1}\ar[dr]_{s\otimes 1_P}\ar[ur]^{1_P\otimes s}\\
&(A[m])\otimes_AP\ar[r]^(.6){l_{P,m}}&P[m]}
\end{equation}
We also consider the following condition.
\begin{enumerate}[\rm(R4)]
\item The map $Z(A)\to\End_{\DDD(k)}(P\otimes_{A^{\mathrm{e}}}A)$ given by $a\mapsto1_P\otimes(a1_A )$ is injective.
\end{enumerate}
The aim of this section is to prove the following result.

\begin{theorem}\label{trivial extension is periodic: second part}
Let $A$ be a finite-dimensional algebra over a field $k$. Assume that $P\in\CCC(A^{\mathrm{e}})$ satisfies the three conditions {\rm(R1)-(R3)} above. Then the element $z\in Z(A)^\times$ given in \eqref{define tuz} satisfies
\[z^\ell=(-1)^m.\]
Moreover, if the condition {\rm(R4)} above is also satisfied, then $z^{\ell+1}=1$ holds. Thus, we have
\[z=(-1)^m\ \mbox{ and }\ (-1)^{(\ell+1)m}=1\mbox{ in }k,\]
that is, at least one of $\ell+1$ and $m$ is even, or ${\rm char} k=2$.
\end{theorem}

\begin{remark}
We note that the condition (R4) can be weakened as follows.
\begin{enumerate}[\rm(R4$'$)]
\item The map $Z(A)\to\End_{\DDD(Z(A))}(P\otimes_{A^{\mathrm{e}}}A)$ given by $a\mapsto1_P\otimes(a1_A )$ is injective.
\end{enumerate}
We prove Theorem \ref{trivial extension is periodic: second part} using (R4) since the proof using (R4$'$) becomes more technical and we only need Theorem \ref{trivial extension is periodic: second part} using (R4)  in later sections.
\end{remark}

The proof of the first equality $z^\ell=(-1)^m$ is divided into two lemmas.

\begin{lemma}\label{z^ell}
The following diagram commutes in $\CCC(A^{\mathrm{e}})$. 
\[\xymatrix@R0.5em@C4em{
P^{\otimes\ell}[m]\ar[dd]_{(-1)^m1_{P^\ell[m]}}&P^{\otimes\ell}\otimes_A(A[m])\ar[l]_{r_{P^\ell,m}}\ar[dr]^{s\otimes 1_{A[m]}}\\
&&(A[m])\otimes_A(A[m])\\
P^{\otimes\ell}[m]\ar[r]^{l_{P^\ell,m}^{-1}}&(A[m])\otimes_AP^{\otimes\ell}\ar[ru]_{1_{A[m]}\otimes s}}\]
\end{lemma}

\begin{proof}
Let $x$ be a homogeneous element in $P^{\otimes\ell}$. The element $x\otimes 1\in P^{\otimes\ell}\otimes_A(A[m])$ is sent to $s(x)\otimes1$ by $s\otimes 1_{A[m]}$, and sent to $(-1)^{m+|x|m}1\otimes s(x)$ by the composition of the other four maps. We claim that the equality
\[s(x)\otimes1=(-1)^{m+|x|m}1\otimes s(x)\ \mbox{ holds in }\ (A[m])\otimes_A(A[m]).\]
In fact, if $|x|=-m$, then $(-1)^m = (-1)^{|x|m}$ holds and the equality follows from the defining property of tensoring over $A$. Otherwise, $s(x)=0$ holds as $A[m]$ is concentrated in degree $-m$. This completes the proof.
\end{proof}

For each $\ell\ge0$ and $a\in\Z(A)$, the equality $\lambda_{P^{\ell},a}=\rho_{P^{\ell},a}$ in $\DDD(A^{\mathrm{e}})$ can be shown inductively:
\[
\lambda_{P^{\ell+1},a}\stackrel{\eqref{1a=a1}}{=}\lambda_{P^{\ell},a}\otimes_A1_P=\rho_{P^{\ell},a}\otimes_A 1_P\stackrel{\eqref{1a=a1}}{=}1_{P^{\ell}}\otimes_A \lambda_{P,a}\stackrel{{\rm(R3)}}{=}1_{P^{\ell}}\otimes_A \rho_{P,a}\stackrel{\eqref{1a=a1}}{=}\rho_{P^{\ell+1},a}.
\]
Thus, for each $i\in\Z$, we have $a1_{P^\ell[i],a}:=\lambda_{P^\ell[i],a}=\rho_{P^\ell[i],a}\in\End_{\DDD(A^{\mathrm{e}})}(P^\ell[i])$, see \eqref{a1}.

\begin{lemma}\label{(-1)^m}
The following diagram commutes in $\DDD(A^{\mathrm{e}})$.
\[\xymatrix@R0.5em@C4em{
P^{\otimes\ell}[m]\ar[dd]_{z^\ell 1_{P^\ell[m]}}&P^{\otimes\ell}\otimes_A(A[m])\ar[l]_{r_{P^\ell,m}}\ar[dr]^{s\otimes 1_{A[m]}}\\
&&(A[m])\otimes_A(A[m])\\
P^{\otimes\ell}[m]\ar[r]^{l_{P^\ell,m}^{-1}}&(A[m])\otimes_AP^{\otimes\ell}\ar[ur]_{1_{A[m]}\otimes s}
}\]
\end{lemma}

\begin{proof}
For simplicity, we omit the subscript of the identity map $1_X$, and write $1^i:=1\otimes\cdots\otimes 1$ the $i$-fold tensor map.  By \eqref{f'g'fg}, we have a commutative diagram in $\CCC(A^{\mathrm{e}})$:
\[\xymatrix@R1em{
&P^{\otimes\ell}\otimes_A(A[m])\ar[dr]^{s\otimes 1}\\
P^{\otimes2\ell}\ar[ur]^{1^\ell\otimes s}\ar[dr]_{s\otimes 1^\ell}\ar[rr]^{s\otimes s}&&(A[m])\otimes_A(A[m])\\
&(A[m])\otimes_AP^{\otimes\ell}\ar[ur]_{1\otimes s}.
}\]
Applying $P\otimes_A-$ and $-\otimes_AP$ repeatedly to the left part of \eqref{define tuz}, we obtain the following commutative diagram in $\DDD(A^{\mathrm{e}})$.
\[\xymatrix@R2em@C8em{
&P^{\otimes\ell}\otimes_A(A[m])\ar[d]^{1^{\ell-1}\otimes t}\\
&P^{\otimes\ell-1}\otimes_A(A[m])\otimes_AP\ar[d]^{1^{\ell-2}\otimes t\otimes 1}\\
P^{\otimes2\ell}\ar@/^/[uur]^{1^\ell\otimes s}\ar[ur]^{1^{\ell-1}\otimes s\otimes 1}\ar[dr]^{1\otimes s\otimes 1^{\ell-1}}\ar@/_/[ddr]^{s\otimes 1^\ell}&\cdots\ar[d]^{1\otimes t\otimes 1^{\ell-2}}\\
&P\otimes_A(A[m])\otimes_AP^{\otimes\ell-1}\ar[d]^{t\otimes 1^{\ell-1}}\\
&(A[m])\otimes_AP^{\otimes\ell}
}\]
We denote by $t^{(\ell)}:P^{\otimes\ell}\otimes_A(A[m])\to (A[m])\otimes_AP^{\otimes\ell}$ the composition of the vertical morphisms. Then the commutativity of the two diagrams above implies that the diagram
\begin{equation}\label{t^ell}
\vcenter{\xymatrix@R.5em{
P^{\otimes\ell}\otimes_A(A[m])\ar[dr]^{s\otimes 1}\ar[dd]^{t^{(\ell)}}\\
&(A[m])\otimes_A(A[m])\\
(A[m])\otimes_AP^{\otimes\ell}\ar[ur]_{1\otimes s}
}}\end{equation}
also commutes in $\DDD(A^{\mathrm{e}})$.

On the other hand, for $X,Y,Z\in\DDD(A^{\mathrm{e}})$ and $n\in\Z$, using \eqref{[m] and [n]} twice, we obtain a canonical isomorphism in $\CCC(A^{\mathrm{e}})$:
\begin{align*}
b_{XYZ,m}:X\otimes_A(Y[n])\otimes_AZ &\to (X\otimes_AY\otimes_AZ)[n]\\
x\otimes y\otimes z & \mapsto (-1)^{|x|n}x\otimes y\otimes z.
\end{align*}
Now we have the following commutative diagram  in $\DDD(A^{\mathrm{e}})$ for each $1\le i\le \ell$, where the left square commutes by the right part of \eqref{define tuz}, the right square commutes by \eqref{1a=a1} and \eqref{fa=af}, and the top and bottom triangles commute by direct calculations:
\[\xymatrix@C3em{
P^{\otimes i}\otimes_A(A[m])\otimes_AP^{\otimes\ell-i}\ar[d]^{1^{i-1}\otimes t\otimes 1^{\ell-i}}\ar[rr]^{1^{i-1}\otimes r_{P,m}\otimes 1^{\ell-i}}\ar@(u,u)[rrrr]^{b_{P^{i}AP^{\ell-i},m}}&\ar@{}[d]^(.4){\eqref{define tuz}}&
P^{\otimes i-1}\otimes_A(P[m])\otimes_AP^{\otimes \ell-i}\ar[d]^{1^{i-1}\otimes z1\otimes 1^{\ell-i}}\ar[rr]^(.65){b_{P^{i-1}PP^{\ell-i},m}}&\ar@{}[d]|(.4){\eqref{1a=a1}\eqref{fa=af}}&P^{\otimes\ell}[m]\ar[d]^{z1}\\
P^{\otimes i-1}\otimes_A(A[m])\otimes_AP^{\otimes\ell-i+1}\ar[rr]^{1^{i-1}\otimes l_{P,m}\otimes 1^{\ell-i}}
\ar@(d,d)[rrrr]^{b_{P^{i-1}AP^{\ell-i+1},m}}&&
P^{\otimes i-1}\otimes_A(P[m])\otimes_AP^{\otimes\ell-i}\ar[rr]^(.65){b_{P^{i-1}PP^{\ell-i},m}}&&P^{\otimes\ell}[m].
}\]
Combining these $\ell$ diagrams, we have a commutative diagram in $\DDD(A^{\mathrm{e}})$:
\[\xymatrix@C3em{
P^{\otimes\ell}\otimes_A(A[m])\ar[d]^{t^{(\ell)}}\ar[rr]^(.6){b_{P^\ell AA,m}=r_{P^\ell,m}}&&P^{\otimes\ell}[m]\ar[d]^{z^\ell1}\\
(A[m])\otimes_AP^{\otimes\ell}\ar[rr]^(.6){b_{AAP^\ell,m}=l_{P^\ell,m}}&&P^{\otimes\ell}[m].
}\]
This together with \eqref{t^ell} shows the assertion.
\end{proof}

We are ready to prove the first equality in Theorem \ref{trivial extension is periodic: second part}.

\begin{proof}[Proof of $z^\ell=(-1)^m$ in Theorem \ref{trivial extension is periodic: second part}]
By comparing the commutative diagrams in Lemmas \ref{z^ell} and \ref{(-1)^m} in $\DDD(A^{\mathrm{e}})$, we obtain the equality $z^\ell 1_{P^\ell[m]}=(-1)^m1_{P^\ell[m]}$ in $\End_{\DDD(A^{\mathrm{e}})}(P^{\otimes\ell}[m])$. Since $\End_{\DDD(A^{\mathrm{e}})}(P^{\otimes\ell}[m])\simeq\End_{\DDD(A^{\mathrm{e}})}(A[2m])\simeq Z(A)$, it follows from condition (R3) that we have $z^\ell=(-1)^m$ as desired.
\end{proof}

We now turn to the proof of 
the second equality $z^{\ell+1}=1$ in Theorem~\ref{trivial extension is periodic: second part}.
Using \eqref{koszul sign} repeatedly, we obtain an automorphism $\widetilde{\rho}\in\End_{\CCC(k)}(P\otimes_k\cdots\otimes_kP)$, where $P$ is tensored $\ell+1$ times, given by
\[\widetilde{\rho}(x_0\otimes x_1\otimes\cdots\otimes x_{\ell}):=(-1)^{(|x_0|+\cdots+|x_{\ell-1}|)|x_{\ell}|}x_{\ell}\otimes x_0\otimes\cdots\otimes x_{\ell-1}.\]
Note that $P^{\ell+1}\otimes_{A^{\mathrm{e}}}A$ is a quotient of $P\otimes_k\cdots\otimes_kP$ (tensoring $\ell+1$ times) by the subcomplex generated by
\begin{align*}
&(x_1\otimes\cdots\otimes x_ia\otimes x_{i+1}\otimes\cdots x_n)-(x_1\cdots\otimes x_i\otimes ax_{i+1}\otimes\cdots\otimes x_n)\text{ for all }1\le i\le n-1,\\
\text{and } &(x_1\otimes\cdots\otimes x_{n-1}\otimes x_na)-(ax_1\otimes x_2\otimes\cdots\otimes x_n),
\end{align*}
with $x_i\in P$ and $a\in A$.  Hence, $\widetilde{\rho}$ induces an automorphism $\rho\in\End_{\CCC(k)}(P^{\otimes\ell+1}\otimes_{A^{\mathrm{e}}}A)$ given by
\[\rho\big((x_0\otimes x_1\otimes\cdots\otimes x_{\ell})\otimes1_A\big):=(-1)^{(|x_0|+\cdots+|x_{\ell-1}|)|x_{\ell}|}(x_{\ell}\otimes x_0\otimes\cdots\otimes x_{\ell-1})\otimes1_A.\]
An easy calculation of signs shows that we have
\begin{equation}\label{rotate}
\rho^{\ell+1}=1_{P^{\ell+1}\otimes_{A^{\mathrm{e}}}A}\ \text{ in }\ \CCC(k).
\end{equation}

Regarding $A$ as an $A^{\mathrm{e}}$-module, we have a \emph{Hochschild functor}
\[ -\otimes_{A^{\mathrm{e}}}A:\CCC(A^{\mathrm{e}})\to\CCC(k). \]
Note in particular that for $X\in\CCC(A^{\mathrm{e}})$, we have \[
xa\otimes1_A=x\otimes a=ax\otimes1_A \text{ in } X\otimes_{A^{\mathrm{e}}}A
\]
for all $x\in X$ and $a\in A$. 

We denote the derived Hochschild functor by
\[c:=-\Lotimes_{A^{\mathrm{e}}}A:\DDD(A^{\mathrm{e}})\to\DDD(k).\]
We have the following key observation.

\begin{lemma}\label{rho=z}
$\rho=z1_{c(P^{\ell+1})}$ holds in $\End_{\DDD(k)}(c(P^{\otimes\ell+1}))$.
\end{lemma}

\begin{proof}
The following diagram commutes in $\DDD(A^{\mathrm{e}})$ by \eqref{define tuz} and \eqref{fa=af}:
\[\xymatrix{
&&P\otimes_A(A[m])\ar[rr]^{r_{P,m}}\ar@{}[rd]|{\eqref{define tuz}}&&P[m]\ar[d]^{z1}\\
P^{\otimes\ell+1}\ar[rru]^{1\otimes s}\ar[rr]^{s\otimes 1}\ar[d]^{z^{-1}1}&&(A[m])\otimes_AP\ar[rr]^{l_{P,m}}\ar@{}[d]|{\eqref{fa=af}}&&P[m]\ar[d]^{z^{-1}1}\\
P^{\otimes\ell+1}\ar[rr]^{s\otimes 1}&&(A[m])\otimes_AP\ar[rr]^{l_{P,m}}&&P[m].
}\]
Applying $c$, we have a commutative diagram in $\DDD(k)$:
\begin{equation}\label{rho=z 1}
\xymatrix{
c(P^{\otimes\ell+1})\ar[rr]^{c(1\otimes s)}\ar[d]^{z^{-1}1}&&c(P\otimes_A(A[m]))\ar[rr]^{c(r_{P,m})}&&c(P[m])\\
c(P^{\otimes\ell+1})\ar[rr]^{c(s\otimes 1)}&&c((A[m])\otimes_AP)\ar[rru]_{c(l_{P,m})}.
}\end{equation}
On the other hand, we claim that the diagram
\begin{equation}\label{rho=z 2}
\xymatrix{
c(P^{\otimes\ell+1})\ar[rr]^{c(1\otimes s)}&&c(P\otimes_A(A[m]))\ar[rr]^{c(r_{P,m})}&&c(P[m])\\
c(P^{\otimes\ell+1})\ar[rr]^{c(s\otimes 1)}\ar[u]^{\rho}&&c((A[m])\otimes_AP)\ar[rru]_{c(l_{P,m})}.
}\end{equation}
commutes in $\DDD(k)$. 
Fix $(x\otimes y)\otimes1_A\in c(P^{\otimes\ell+1})$, where $x\in P^{\otimes\ell}$ and $y\in P$ are homogeneous. The composition of maps in the second row sends $(x\otimes y)\otimes1_A$ to $s(x)y\otimes 1_A\in c(P[m])$.
The composition of the two maps in the first row sends $\rho((x\otimes y)\otimes1_A)=(-1)^{|x||y|}(y\otimes x)\otimes1_A$ to
\begin{align*}
(-1)^{|x||y|}r_{P,m}(y\otimes s(x))\otimes1_A&=(-1)^{|x||y|+m|y|}ys(x)\otimes1_A\\
&=ys(x)\otimes1_A=y\otimes s(x)=s(x)y\otimes1_A.
\end{align*}
In the second equality, we can get rid of the sign for the following reason.  Indeed, If $|x|$ is $-m$, then the sign is $1$. Otherwise, $s(x)=0$ since $A[m]$ is concentrated in degree $-m$. 
Note also that the final equality follows from the defining property of the Hochschild functor. Thus, \eqref{rho=z 2} commutes.

Comparing \eqref{rho=z 1} and \eqref{rho=z 2}, we obtain $\rho=z1_{c(P^{\ell+1})}$.
\end{proof}

We are ready to prove the second equality in Theorem \ref{trivial extension is periodic: second part}.

\begin{proof}[Proof of $z^{\ell+1}=1$ in Theorem \ref{trivial extension is periodic: second part}]
Using Lemma \ref{rho=z} and \eqref{rotate}, we have \[
z^{\ell+1}1_{c(P^{\ell+1})}\stackrel{\ref{rho=z}}{=}\rho^{\ell+1}\stackrel{\eqref{rotate}}{=}1_{c(P^{\ell+1})}\]
in $\End_{\DDD(k)}(c(P^{\ell+1}))$.
The claimed equality $z^{\ell+1}=1$ now follows from $c(P^{\ell+1})\simeq c(P)[m]$ and the condition (R4).
\end{proof}

\section{Fractionally Calabi--Yau algebras have periodic trivial extensions}\label{sec:fCY to per}

The aim of this section is to prove the following, principal theorem of this paper.

\begin{theorem}\label{trivial extension is periodic}
Let $A$ be a finite-dimensional algebra over a field $k$ such that $\gldim A<\infty$ and $A/\rad A$ is a separable $k$-algebra. Assume that $A$ is $\frac{m}{\ell}$-Calabi--Yau.
\begin{enumerate}[\rm(a)]
\item For $\varphi\in\Aut_k^{\Z}(T(A))$ given by $\varphi(a,f)=(a,(-1)^{\ell+m}f)$ for $(a,f)\in A\oplus DA=T(A)$, we have an isomorphism
\[\Omega^{\ell+m}_{T(A)^{\mathrm{e}}}(T(A))\simeq{}_\varphi T(A)_1\ \mbox{ in }\mod T(A)^{\mathrm{e}}.\]
In particular, $T(A)$ is $2(\ell+m)$-periodic.
\item If $\CYdim A=(m,\ell)$, then the minimal period of $T(A)$ is $\ell+m$ if $(-1)^{\ell+m}=1$ in $k$, and $2(\ell+m)$ otherwise.
\item At least one of $\ell+1$ and $m$ is even, or ${\rm char} k=2$.
\end{enumerate}
\end{theorem}

Our results immediately give the following partial answer to the Periodicity conjecture.

\begin{corollary}\label{finite out}
  Let $A$ be a finite-dimensional $k$-algebra such that $A/\rad A$ is separable over the field $k$.
  \begin{enumerate}[\rm(a)]
  \item
    The trivial extension algebra $T(A)$ is periodic if and only if $A$ has finite global dimension and is fractionally Calabi--Yau.
    \label{finite out1}
  \item 
    If the outer automorphism group of $A$ is finite, then $T(A)$ is periodic if and only if it is twisted periodic.
    \label{finite out2}
  \end{enumerate}
\end{corollary}

\begin{proof}
(a) The ``if'' part follows from Theorem~\ref{trivial extension is periodic}(a). The ``only if'' part follows from Propositions~\ref{periodic to fCY} and \ref{fin gl}.

(b) By Theorem~\ref{main theorem 2 again}, $T(A)$ is twisted periodic if and only if $A$ is twisted fractionally Calabi--Yau. Under the assumption that $\Out(A)$ is finite, this is equivalent to $A$ being fractionally Calabi--Yau, which in turn is equivalent to the periodicity of $T(A)$, by \eqref{finite out1}.
\end{proof}

To prove Theorem~\ref{trivial extension is periodic}, let $P\in\CCC(A^{\mathrm{e}})$ be a projective resolution of the $A^{\mathrm{e}}$-module $DA$, where
\[P=[\cdots\xrightarrow{d^{-3}} P^{-2}\xrightarrow{d^{-2}} P^{-1}\xrightarrow{d^{-1}}P^0\to0\cdots].\]

\begin{lemma}\label{lem:conditions OK}
$P\in\CCC(A^{\mathrm{e}})$ satisfies the four conditions {\rm(R1)-(R4)} in the previous section.
\end{lemma}

\begin{proof}
(R1) is clear, (R2) holds since $A$ is $\frac{m}{\ell}$-Calabi--Yau, and (R3) holds since $\End_{\DDD(A^{\mathrm{e}})}(P)\simeq\End_{\DDD(A^{\mathrm{e}})}(DA)\simeq\End_{\DDD(A^{\mathrm{e}})}(A)\simeq Z(A)$. It remains to show (R4). Since
\[c(P)\simeq DA\Lotimes_{A^{\mathrm{e}}}A\simeq D\RHom_{A^{\mathrm{e}}}(A,A),\]
we have $H^0(c(P))=D\End_{A^{\mathrm{e}}}(A)=DZ(A)$. Thus the composition $Z(A)\to\End_{\DDD(k)}(c(P))\xrightarrow{H^0}\End_{k}(DZ(A))$ sends $a\in Z(A)$ to $a1$. Hence it is injective, so (R4) is satisfied.
\end{proof}

In particular, we can define $z\in Z(A)^\times$ by the commutative diagram \eqref{define tuz}. The crucial part in the proof of Theorem \ref{trivial extension is periodic} is the following result, which is independent of Theorem \ref{trivial extension is periodic: second part}.

\begin{proposition}\label{trivial extension is periodic: first part}
In the setting of Theorem \ref{trivial extension is periodic}, we define $\varphi\in\Aut_k^{\Z}(T(A))$ by $\varphi(a,f)=(a,(-1)^{\ell}zf)$ for $(a,f)\in A\oplus DA=T(A)$. Then we have an isomorphism
\[\Omega^{\ell+m}_{T(A)^{\mathrm{e}}}(T(A))\simeq{}_\varphi T(A)_1\ \mbox{ in }\mod T(A)^{\mathrm{e}}.\]
\end{proposition}

For the rest of Section~\ref{sec:fCY to per}, let $B:=T(A)$. To prove Proposition \ref{trivial extension is periodic: first part}, notice that
\[C:=A\oplus P\in\CCC(A^{\mathrm{e}})\]
has a natural structure of a dg $k$-algebra such that $(a,f)\cdot (b,g)=(ab, ag+fb)$ for $(a,f),(b,g)\in A\oplus P=C$. Then we have a quasi-isomorphism
\[C\to H^0(C)=B\]
of dg $k$-algebras. We denote by $\CCC(C^{\mathrm{e}})$ the category of dg $C^{\mathrm{e}}$-modules. For $i\ge0$, let
\begin{eqnarray*}
Q_i:=C\otimes_AP^{\otimes i}\otimes_AC\in\CCC(C^{\mathrm{e}}).
\end{eqnarray*}
For $i >0$, define a morphism $f_i:Q_i\to Q_{i-1}$ in $\CCC(C^{\mathrm{e}})$ by
\begin{align*}
f_i(c_0\otimes c_1\otimes\cdots\otimes c_i\otimes c_{i+1})&=\sum_{j=0}^i(-1)^jc_0\otimes c_1\otimes\cdots\otimes c_jc_{j+1}\otimes\cdots\otimes c_i\otimes c_{i+1}\\
&=c_0c_1\otimes c_2\otimes\cdots\otimes c_{i+1}+(-1)^i
c_0\otimes \cdots\otimes c_{i-1}\otimes c_ic_{i+1},
\end{align*}
where $c_0,c_{i+1}\in C$ and other $c_j$'s are in $P$.
This, together with $f_0:Q_0\to C,\:c_0\otimes c_1\mapsto c_0c_1$, gives a relative bar resolution, that is, a complex
\[\cdots\xrightarrow{f_2}Q_1\xrightarrow{f_1}Q_0\xrightarrow{f_0}C\ \mbox{ in }\ \CCC(C^{\mathrm{e}})\]
whose total dg $C^{\mathrm{e}}$-module is acyclic.

We truncate this relative bar resolution using the dg $C^{\mathrm{e}}$-module
\[N:=P^{\otimes\ell}\oplus P^{\otimes\ell+1},\]
whose $C^{\mathrm{e}}$-action is given by $(x,y)(a,p):=(xa,ya+x\otimes p)$ and $(a,p)(x,y):=(ax,ay+p\otimes x)$ for $(x,y)\in P^{\otimes\ell}\oplus P^{\otimes\ell+1}=N$ and $(a,p)\in A\oplus P=C$.  We denote by
\[\sigma:C\to C\]
the automorphism of the dg $k$-algebra $C$ given by $\sigma(a,p)=(a,(-1)^\ell p)$ for $(a,p)\in A\oplus P=C$. We denote by ${}_\sigma N$ the dg $C^{\mathrm{e}}$-module whose left action of $C$ is twisted by $\sigma$, that is, $(a,p)(x,y):=(ax,ay+(-1)^\ell p\otimes x)$.
We define a morphism
\begin{align*}
g_\ell:{}_\sigma N\to Q_{\ell-1}\ \mbox{ in }\ \CCC(A^{\mathrm{e}})\ \mbox{ by }\ g_\ell(x,y)=(0,(-1)^\ell 1_A\otimes x,x\otimes 1_A,y),
\end{align*}
where $Q_{\ell-1}$ is the direct sum $(A\otimes_AP^{\otimes\ell-1}\otimes_AA)\oplus(A\otimes_AP^{\otimes\ell-1}\otimes_AP)\oplus(P\otimes_AP^{\otimes\ell-1}\otimes_AA)\oplus(P\otimes_AP^{\otimes\ell-1}\otimes_AP)$.

The first key ingredient of the proof is the following.

\begin{lemma}\label{truncate bar resolution}
The map $g_\ell$ is a morphism in $\CCC(C^{\mathrm{e}})$, satisfying $f_{\ell-1}\circ g_\ell=0$. Moreover, the total dg $C^{\mathrm{e}}$-module of
\begin{equation}\label{bar cpx}{}_\sigma N\xrightarrow{g_\ell}Q_{\ell-1}\xrightarrow{f_{\ell-1}}\cdots\xrightarrow{f_1}Q_0\xrightarrow{f_0}C\ \mbox{ in }\ \CCC(C^{\mathrm{e}})\end{equation}
is acyclic.
\end{lemma}

\begin{proof}
It is easy to check that $g_\ell$ is a morphism of graded $C^{\mathrm{e}}$-modules, thanks to the twist $\sigma$. As morphisms in $\CCC(A^{\mathrm{e}})$, $g_\ell$ and maps $f_i$ can be written as follows, where $\cdot:=\otimes_A$, $P^i:=P^{\otimes i}$ and $\epsilon:=(-1)^\ell$.
\[\xymatrix@R0em@C2em{
P^\ell\ar[dr]|(.35){\epsilon}\ar[ddr]|(.35)1&A\cdot P^{\ell-1}\cdot A\ar[dr]|(.65){-\epsilon}\ar[ddr]|(.65)1&\cdots\ar[dr]|(.35){-1}\ar[ddr]|(.35)1&A\cdot P^2\cdot A\ar[dr]|{1}\ar[ddr]|1&A\cdot P\cdot A\ar[dr]|{-1}\ar[ddr]|1&A\cdot A\ar[dr]|1\\
&A\cdot P^{\ell-1}\cdot P\ar[ddr]|(.65)1&\cdots\ar[ddr]|(.35)1&A\cdot P^2\cdot P\ar[ddr]|1&A\cdot P\cdot P\ar[ddr]|1&A\cdot P\ar[ddr]|1&A\\
P^{\ell+1}\ar[dr]|(.35)1&P\cdot P^{\ell-1}\cdot A\ar[dr]|(.65){-\epsilon}&\cdots\ar[dr]|(.35){-1}&P\cdot P^2\cdot A\ar[dr]|1&P\cdot P\cdot A\ar[dr]|{-1}&P\cdot A\ar[dr]|1\\
&P\cdot P^{\ell-1}\cdot P&\cdots&P\cdot P^2\cdot P&P\cdot P\cdot P&P\cdot P&P.}\]
Therefore, $g_\ell$ is a morphism in $\CCC(C^{\mathrm{e}})$, satisfying $f_{\ell-1}\circ g_\ell=0$, and the total dg module is contractible as a dg $A^{\mathrm{e}}$-module, and hence acyclic.
\end{proof}

Let
\[\Db(C^{\mathrm{e}}):=\{X\in\DDD(C^{\mathrm{e}})\mid\dim H(X)<\infty\}\ \mbox{ and }\ \DDD_{\sg}(C^{\mathrm{e}}):=\Db(C^{\mathrm{e}})/\per C^{\mathrm{e}}.\]
From Lemma~\ref{truncate bar resolution}, we obtain the following result.

\begin{lemma}\label{C and N}
We have an isomorphism $C\simeq {}_\sigma N[\ell]$ in $\DDD_{\sg}(C^{\mathrm{e}})$.
\end{lemma}

\begin{proof}
Since $\gldim A$ is finite and $A/{\rm rad} A$ is a separable $k$-algebra, it follows that $\gldim A^{\mathrm{e}}$ is also finite. Thus $P^{\otimes i}\in\per A^{\mathrm{e}}$ and hence $Q_i\in\per C^{\mathrm{e}}$ and $Q_i\simeq 0$ in $\DDD_{\sg}(C^{\mathrm{e}})$.

As usual, for objects $X,Y$ in a triangulated category $\TT$, we write
\[X*Y:=\{Z\in\TT\mid \mbox{there exists a triangle $X\to Z\to Y\to X[1]$ in $\TT$}\}.\]
The total dg module of \eqref{bar cpx} is in $C*Q_0[1]*\cdots*Q_{\ell-1}[\ell]*{}_\sigma N[\ell+1]$ and is isomorphic to the zero object in $\DDD(C^{\mathrm{e}})$ by Lemma~\ref{truncate bar resolution}.
Since each $Q_i\simeq 0$ in $\DDD_{\sg}(C^{\mathrm{e}})$, the zero object is contained in the subcategory $C*{}_\sigma N[\ell+1]$ of $\DDD_{\sg}(C^{\mathrm{e}})$. Thus there is a triangle $C\to 0\to{}_\sigma N[\ell+1]\to C[1]$ in $\DDD_{\sg}(C^{\mathrm{e}})$, and hence ${}_\sigma N[\ell]\simeq C$.
\end{proof}

The quasi-isomorphism $C\to H^0(C)=B$ gives a quasi-isomorphism $C^{\mathrm{e}}\to H^0(C^{\mathrm{e}})=B^{\mathrm{e}}$ and hence we have equivalences 
\[F:\DDD(C^{\mathrm{e}})\simeq\DDD(B^{\mathrm{e}})\ \mbox{ and }\ \DDD_{\sg}(C^{\mathrm{e}})
\simeq\DDD_{\sg}(B^{\mathrm{e}}).\]
We denote by
\[\tau:B\to B\]
an automorphism of the $k$-algebra $B$ given by $\tau(a,f)=(a,zf)$ for $(a,f)\in A\oplus DA=B$.

\begin{lemma}\label{N and B}
We have $F(C)\simeq B$, and $F(N)\simeq{}_\tau B[m]$ in $\DDD(B^{\mathrm{e}})$.
\end{lemma}

To prove this, recall that, for each integer $\ell$, a full subcategory
\[\DDD^{\ell}(C^{\mathrm{e}}):=\{X\in\Db(C^{\mathrm{e}})\mid H^i(X)=0\mbox{ for all }i\neq\ell\}\subset\DDD(C^{\mathrm{e}})\]
is the heart of a shifted standard t-structure of $\DDD(C^{\mathrm{e}})$, and we have an equivalence (e.g.\ \cite[Proposition 4.8]{IY})
\begin{equation}\label{describe heart}
H^{\ell}:\DDD^{\ell}(C^{\mathrm{e}})\to\mod B^{\mathrm{e}}.
\end{equation}

\begin{proof}[Proof of Lemma \ref{N and B}]
Clearly $F(C)\simeq B$ holds. 

We will show that $F(N)\simeq{}_\tau B[m]$. 
By (R2), there are isomorphisms 
$s:P^{\otimes\ell}\stackrel{\sim}{\to} A[m]$ 
and $s\otimes1_P:P^{\otimes\ell+1}\stackrel{\sim}{\to} A[m]\otimes_AP$ in $\DDD(A^{\mathrm{e}})$, and it follows that $N\in\DDD^{-m}(C^{\mathrm{e}})$.
Thanks to the equivalence \eqref{describe heart}, it suffices to show that $H^{-m}(N)$ is isomorphic to ${}_\tau B$ as $B^{\mathrm{e}}$-module. 

Consider the morphism
\[u:=s\oplus(s\otimes 1_P):N=P^{\otimes\ell}\oplus P^{\otimes\ell+1} \to A[m]\oplus(A[m]\otimes_AP) = A[m]\otimes_AC\ \mbox{ in }\ \CCC(A^{\mathrm{e}})\]
and the induced morphism
\[H^{-m}(u):H^{-m}(N)\to H^{-m}(C[m])=B\ \mbox{ in }\ \mod A^{\mathrm{e}}.\]
Since $N$ is a dg $C^{\mathrm{e}}$-module, $H^{-m}(N)$ has a natural $B^{\mathrm{e}}$-module structure. To prove our claim, it suffices to show the following.
\begin{enumerate}
\item $H^{-m}(u)$ is a morphism of right $B$-modules.
\item If we twist the left action of $B$ on $B$ by $\tau$, then $H^{-m}(u)$ is a morphism of left $B$-modules.
\end{enumerate}
The claim (1) is clear. In fact, since $u=s\otimes1_C:N=P^{\otimes\ell}\otimes_C C\to A[m]\otimes_C C\simeq C[m]$ is a morphism of right dg $C$-modules, $H^{-m}(u)$ is a morphism of right $B$-modules.

It remains to show (2). It suffices to check that $H^{-m}(u)$ commutes with the left action of $DA\subset B$ after twisting by $\tau$, that is, for each $p\in Z^0(P)$, the diagram
\begin{equation}\label{show commutative}
\xymatrix{H^{-m}(P^{\otimes\ell})\ar[rr]^{H^{-m}(s)}\ar[d]_{H^{-m}(p\otimes-)}&& H^{-m}(A[m]) \ar@{=}[rrr] &&&A\ar[d]^{zH^0(p\cdot)}\\
H^{-m}(P^{\otimes\ell+1})\ar[rr]^(.45){H^{-m}(s\otimes1_P)}&&H^{-m}(A[m]\otimes_AP)\ar[rr]^(.55){H^{-m}(l_{P,m})}&&H^{-m}(P[m])\ar[r]^(.6)\sim&DA}
\end{equation}
commutes, where the map $H^0(p\cdot)$ is induced from the left multiplication $(p\cdot):A\to P$.

By \eqref{define tuz}, we have a commutative diagram in $\DDD(A^{\mathrm{e}})$:
\[\xymatrix@R0.5em@C3em{
&P\otimes_A(A[m])\ar[r]^(.6){r_{P,m}}&P[m]\ar[r]^\sim&DA[m]\ar[dd]^{z1_{DA[m]}}\\
P^{\otimes\ell+1}\ar[ur]^(.45){1_P\otimes s}\ar[dr]_(.45){s\otimes1_P}\\
&A[m]\otimes_AP\ar[r]^(.6){l_{P,m}}&P[m]\ar[r]^\sim&DA[m].
}\]
Applying $H^{-m}$ yields the following commutative diagram in $\mod A^{\mathrm{e}}$:
\begin{equation}\label{know commutative}
\xymatrix@R.5em@C4em{
&H^{-m}(P\otimes_A(A[m]))\ar[r]^(.55){H^{-m}(r_{P,m})}&H^{-m}(P[m])\ar[r]^(.65)\sim&DA\ar[dd]^{z1_{DA}}\\
H^{-m}(P^{\otimes\ell+1})\ar[ur]^(.47){H^{-m}(1_P\otimes s)}\ar[dr]_(.47){H^{-m}(s\otimes1_P)}\\
&H^{-m}(A[m]\otimes_AP)\ar[r]^(.55){H^{-m}(l_{P,m})}&H^{-m}(P[m])\ar[r]^(.65)\sim&DA.
}\end{equation}
Now fix $p\in Z^0(P)$ and $x\in Z^{-m}(P^{\otimes\ell})$.
The image of $x$ through the four lower maps in \eqref{show commutative} clearly equals the image of $p\otimes x$ under the lower composition in \eqref{know commutative}.

For any $i\in \Z$, $Y\in\DDD(A^{\mathrm{e}})$ and $y\in Z^{-i}(Y)$, denote by $[y]\in H^{-i}(Y)$ the cohomology class of $y$. Then the image of $x$ under the upper composition in \eqref{show commutative} is $z[p][s(x)]=z[ps(x)]$. Since $|p|=0$, the image of $p\otimes x$ through the composition in the upper four maps of \eqref{know commutative} is also $z[ps(x)]$.
Thus, the desired commutativity of \eqref{show commutative} follows from that of \eqref{know commutative}.
\end{proof}

Now we are ready to prove Proposition \ref{trivial extension is periodic: first part}.

\begin{proof}[Proof of Proposition \ref{trivial extension is periodic: first part}]
In $\DDD_{\sg}(B^{\mathrm{e}})$, we have an isomorphisms
\[B\stackrel{\ref{N and B}}{\simeq}F(C)\stackrel{\ref{C and N}}{\simeq}F({}_\sigma N[\ell])\stackrel{\ref{N and B}}{\simeq}{}_{\sigma\tau}B[\ell+m]={}_{\varphi}B[\ell+m].\]
Thus $\Omega^{\ell+m}_{B^{\mathrm{e}}}(B)\simeq {}_{\varphi}B$ holds in $\stmod B^{\mathrm{e}}$. Then the assertion follows.
\end{proof}

Last, we prove Theorem~\ref{trivial extension is periodic}.

\begin{proof}[Proof of Theorem \ref{trivial extension is periodic}]
Since Lemma \ref{lem:conditions OK} tells us that all four conditions (R1) to (R4) are satisfied for the projective resolution $P$ of $DA$, the assertions (a) and (c) follow immediately from Proposition \ref{trivial extension is periodic: first part} and Theorem \ref{trivial extension is periodic: second part}. 

It now remains to prove (b).  By (a), $B$ is $(\ell+m)$-periodic if $(-1)^{\ell+m}=1$ in $k$, and $2(\ell+m)$-periodic otherwise. Conversely, assume that $\Omega_{B^{\mathrm{e}}}^n(B)\simeq B$ in $\mod B^{\mathrm{e}}$ for $n\ge1$. Then there exists $a\ge1$ such that $\Omega^n_{B^{\mathrm{e}}}(B)\simeq B(-a)$ in $\mod^{\Z}B^{\mathrm{e}}$, and this gives an isomorphism $\Omega^n\simeq(-a)$ of functors on $\stmod^{\Z}B$.
By Proposition \ref{periodic to fCY}, $A$ is $\frac{n-a}{a}$-Calabi--Yau. Since $\CYdim A=(m,\ell)$, there exists a positive integer $i$ such that $n-a=mi$ and $a=\ell i$. Thus $n=(\ell+m)i$ holds, and we have
\[B\simeq\Omega_{B^{\mathrm{e}}}^{n}(B)\simeq{}_ {\varphi^i}B_1\ \mbox{ in }\ \mod B^{\mathrm{e}}\]
by (a). Thus $\varphi^i=1$ in $\Out_k(B)$, and $(-1)^{(\ell+m)i}=1$ in $k$ by Lemma \ref{inneriffidentity}. Thus $n$ is a multiple of $2(\ell+m)$ if $(-1)^{\ell+m}\neq1$ in $k$.
\end{proof}


\section{Application to self-injective algebras}\label{section:selfinjective}

In this section, we show how our results about periodicity and twisted periodicity can be extended from trivial extensions to more general classes of orbit algebras.

Let $\CC$ be an additive category, and $G$ a group of automorphisms of $\CC$.
The \emph{orbit category} $\CC/G$ has the same objects as $\CC$, and morphism sets
\[\Hom_{\CC/G}(X,Y):=\bigoplus_{g\in G}\Hom_{\CC}(X,gY) \,.\]
The composition of $(a_g)_{g\in G}\in\Hom_{\CC/G}(X,Y)$ and $(b_g)_{g\in G}\in\Hom_{\CC}(Y,Z)$ is given by
\[\left(\sum_{h\in G}(X\xrightarrow{a_h}hY\xrightarrow{h(b_{h^{-1}g})}gZ)\right)_{g\in G}\in\Hom_{\CC/G}(X,Z).\]
Now assume that $\CC$ is Krull--Schmidt, and let $\ind\CC$ be a set of representatives for the isomorphism classes of indecomposable objects in $\CC$. Moreover, assume that $\CC$ is $k$-linear and \emph{locally bounded}, that is, 
\[\sum_{Y\in\ind\CC}\dim_k\Hom_{\CC}(X,Y)<\infty\mbox{ and }\sum_{Y\in\ind\CC}\dim_k\Hom_{\CC}(Y,X)<\infty\]
for all $X\in\ind\CC$.
We call $G$ \emph{admissible} if it acts freely on $\ind\CC$ and $\#(\ind\CC/G)<\infty$. In this case, we can regard $\CC/G$ as a finite-dimensional $k$-algebra.
For example, the repetitive category $\widehat{A}$ of a finite-dimensional $k$-algebra $A$ is locally bounded, and a cyclic group $\langle\phi\rangle\subset\Aut_k(\widehat{A})$ is admissible whenever it acts freely on $\ind\widehat{A}$, by \cite[Lemma~3.4]{DI}.

Let $G$ be a group of automorphisms of $A$.
The induced $G$-action $g\cdot M = g^*(M)$ on $\mod A$ restricts to $\proj A$. The \emph{orbit algebra} $A/G$ is defined as $A/G = \End_{(\proj A)/G}(A)$.
Note that $A/G$ is isomorphic to the skew group algebra $A*G$ (cf.\ \cite[Proposition~2.4]{CM}). We say that $G$ is admissible if its action on $\proj A$ is admissible.

The following result, proved by Dugas under the assumption that $A$ is a split $k$-algebra, holds also in our, somewhat more general, setting.

\begin{proposition}{\cite[Corollary 3.8]{Du2}}\label{periodicity and covering}
Let $\Lambda$ be a basic finite-dimensional algebra over a field $k$ such that $\Lambda/\rad \Lambda$ is a separable $k$-algebra, and $G$ an admissible group of automorphisms of $\Lambda$.
Then $\Lambda$ is periodic if and only if $\Lambda/G$ is periodic.
\end{proposition}

The following results complete the proof of our main Theorems~\ref{main theorem} and \ref{main theorem 2}.

\begin{proposition}\label{application to orbit algebra}
Assume that $A/\rad A$ is a separable $k$-algebra, and $G$ an admissible group of automorphisms of $\widehat{A}$, and $B:=\widehat{A}/G$.
\begin{enumerate}[\rm(a)]
\item The following conditions are equivalent.
\begin{enumerate}[\rm(i)]
\item Each $T(A)$-module has complexity at most one.
\item $T(A)$ is twisted periodic.
\item Each $B$-module has complexity at most one.
\item $B$ is twisted periodic.
\end{enumerate}
\item Assume that $G$ contains $\nu_{\widehat{A}}^\ell$ for some $\ell\ge1$. Then $T(A)$ is periodic if and only if $B$ is periodic.
\end{enumerate}
\end{proposition}

\begin{proof}
(a)
The equivalence (i)$\Leftrightarrow$(ii) was shown in Theorem~\ref{main theorem 2 again}.  

For (i)$\Leftrightarrow$(iii), recall that the push-down functors $\mod\widehat{A}\to\mod T(A)$ and $\mod\widehat{A}\to\mod B$ preserve simple modules and minimal projective resolutions (c.f. \cite[3.5]{DI}). Thus, both (i) and (iii) are equivalent to having $\cx_{\widehat{A}}(S)\leq 1$ for all simple $\widehat{A}$-modules $S$.

Similarly, periodicity of simples for $\widehat{A}$ is equivalent to periodicity of simples for $T(A)$, as well as for $B$. By Proposition~\ref{thm:hani}, these conditions are equivalent to twisted periodicity of $T(A)$ and $B$, respectively, which proves (ii)$\Leftrightarrow$(iv).

(b)
First, note that $T(A\times A')\cong T(A)\times T(A')$. Therefore, by Proposition~\ref{elementary properties}(a), we may assume, without loss of generality, that $A$ is ring-indecomposable. We may also assume that $A$ is basic, by Proposition~\ref{elementary properties}(c).
Recall that $\widehat{A}/\langle\widehat{\nu}^\ell\rangle=T_\ell(A)$, where $\widehat{\nu}=\nu_{\widehat{A}}$.
Now, both $G/\langle\widehat{\nu}^\ell\rangle$ and $\langle\widehat{\nu}\rangle/\langle\widehat{\nu}^\ell\rangle\cong\Z/\ell\Z$ are admissible groups of automorphisms of $T_\ell(A)$, yielding the orbit algebras
\[T_\ell(A)/(G/\langle\widehat{\nu}^\ell\rangle)=B\ \mbox{ and }\ T_\ell(A)/(\Z/\ell\Z)=T(A) \,,\]
respectively.
Using Proposition~\ref{periodicity and covering} twice, it follows that $T_\ell(A)$ is periodic if and only if $T(A)$ is periodic if and only if $B$ is periodic.
\end{proof}

We can now summarize our results about periodicity and twisted periodicity of orbit algebras as follows.

\begin{corollary}\label{application to orbit algebra 2}
Let $A$ be a finite-dimensional algebra over a field $k$ such that $A/\rad A$ is a separable $k$-algebra, and $G$ an admissible group of automorphisms of $\widehat{A}$.
\begin{enumerate}[\rm(a)]
\item The following conditions are equivalent.
\begin{enumerate}[\rm(i)] 
\item $\widehat{A}/G$ is twisted periodic.
\item Each $\widehat{A}/G$-module has complexity at most one.
\item $A$ has finite global dimension and is twisted fractionally Calabi--Yau.
\end{enumerate}
\item If $G$ contains $\nu_{\widehat{A}}^\ell$ for some $\ell\ge1$, then the following conditions are equivalent.
\begin{enumerate}[\rm(i)] 
\item $\widehat{A}/G$ is periodic.
\item $A$ has finite global dimension and is fractionally Calabi--Yau.
\end{enumerate}
\item If the outer automorphism group of $A$ is finite, and $G$ contains $\nu_{\widehat{A}}^\ell$ for some $\ell\ge1$, then $\widehat{A}/G$ is periodic if and only if it is twisted periodic.
\end{enumerate}
\end{corollary}

\begin{proof}
The statement (a) follows from Theorem~\ref{main theorem 2 again} and Proposition~\ref{application to orbit algebra}(a), and (b) follows from Corollary~\ref{finite out} and Proposition~\ref{application to orbit algebra}(b). Statement (c) is immediate from (a) and (b) together with Corollary~\ref{finite out}(b).
\end{proof}


\section{Examples}\label{sec:eg}

In this section, we give examples of (twisted) fractionally Calabi--Yau algebras and (twisted) periodic trivial extension algebras.
The simplest examples are given by symmetric and self-injective algebras, which are $\frac{0}{1}$-Calabi--Yau and twisted $\frac{0}{1}$-Calabi--Yau, respectively.

\subsection{Examples from representation-finite and $d$-representation-finite algebras}

The following count amongst the most fundamental examples of fractionally Calabi--Yau algebras.

\begin{example}\label{ex:dynkin}
Let $kQ$ be the path algebra of a Dynkin quiver, and $h$ the Coxeter number, given as follows:
  \[\begin{array}{|c|c|c|c|c|}
\hline
A_n&D_n&E_6&E_7&E_8\\ \hline
n+1&2(n-1)&12&18&30\\ \hline
\end{array}\]
It is well known, for example from  \cite{MY}, that such an algebra is fractionally Calabi--Yau.  It seems to be folklore, c.f.\ \cite[3.1]{HI}, that $\CYdim kQ=(\frac{h}{2}-1,\frac{h}{2})$ when $Q$ is of type $A_1$, $D_n$ with even $n$ or $E_7, E_8$, and $\CYdim kQ=(h-2,h)$ else.
For completeness, we shall give a proof of this in Section \ref{sec:DynkinCY} (see Theorem~\ref{Dynkin is CY}).
On the other hand, $kQ$ is not fractionally Calabi--Yau when $Q$ is not of Dynkin type, since the Coxeter matrix of $kQ$ is not periodic in this case, see for example \cite[Proposition 3.1]{L2}.
\end{example}

From Example~\ref{ex:dynkin} we get an alternative proof of the following result, which was first obtained in \cite{BBK} by a case-by-case calculation. In contrast, our proof is purely conceptual -- albeit highly technical due to the dg technology involved.

\begin{example}\label{T(kQ) is periodic} 
Let $kQ$ be the path algebra of a Dynkin quiver, and $h$ the Coxeter number of the corresponding Dynkin type. Then the minimal period of $T(kQ)$ is
\[
\begin{cases}
h-1 & \text{if ${\rm char} k=2$, and $Q$ is one of type $A_1$, $D_{2n}$ or $E_7, E_8$;}\\
2h-2 & \text{otherwise.}
\end{cases}
\]
\end{example}
\begin{proof}
By Theorem~\ref{trivial extension is periodic}, the Calabi--Yau dimensions of Dynkin quivers given in Example~\ref{ex:dynkin} directly translates into the minimal periods as claimed.
\end{proof}

Example \ref{ex:dynkin} above admits a generalisation (albeit imperfect) from the point of view of cluster-tilting theory.

Recall that the \emph{Coxeter matrix} $c_A$ of an algebra $A$ of finite global dimension is defined as $c_A:=-U^{-1} U^T$, where $U$ is the Cartan matrix of $A$. When $A$ is $\frac{m}{\ell}$-Calabi--Yau, then $c_A$ is periodic with $c_A^{2 \ell}$ being the identity matrix, see \cite[Lemma 2.9]{Pe}.
Now we give more examples. 

\begin{example}\cite[Theorem 1.1]{HI}\label{eg:fCYdimd}
Let $d\ge1$. Any $d$-representation-finite (Definition~\ref{def:CT}) algebra $A$ with $\gldim A\le d$ is twisted fractionally Calabi--Yau.  More precisely, let $a$ be the number of indecomposable direct summands of a basic $d$-cluster-tilting $A$-module, and $b$ be the number of simple $A$-modules.  Then $A$ is twisted $\frac{m}{\ell}$-Calabi--Yau with $\frac{m}{\ell}=\frac{d(a-b)}{a}$ as a rational number. The case $d=1$ was given in Example \ref{ex:dynkin} with the stronger untwisted property; c.f.{} Question~\ref{remove twist 2}.
\end{example}

We give one other class of fractionally Calabi--Yau algebras arising from higher Auslander--Reiten theory.  For this purpose, we need an abstract result relating the fractional Calabi--Yau property with cluster-tilting subcategories.
We refer to \cite{IO} for any unexplained terminology, as these notions are used only in the following Proposition~\ref{mod CY is CY}.
We recall here that a $k$-linear Hom-finite triangulated category $\TT$ with suspension functor ${}_\TT\Sigma$ and a Serre functor ${}_\TT\Sbb$ is \emph{$\frac{m}{\ell}$-Calabi--Yau} if ${}_\TT\Sbb^\ell \simeq {}_\TT\Sigma^m$ as additive functors, and remark that, with this definition, an Iwanaga--Gorenstein algebra $A$ is $\frac{m}{\ell}$-Calabi--Yau if and only if so is $\per(A)$.

A \emph{$d\Z$-cluster-tilting} subcategory $\UU$ of a triangulated category $\TT$ is a $d$-cluster-tilting subcategory that in addition satisfies ${}_\TT\Sigma^d(\UU)=\UU$ or, equivalently, $\Hom_{\TT}(\UU,{}_\TT\Sigma^i(\UU))=0$ for each $i\in\Z\setminus d\Z$
\cite{JK,Kv}.

For an autoequivalence $\phi$ of $\TT$, we have an induced automorphism $\phi_*$ on the category $\mod \TT$ of finitely presented functors given by precompostion by a chosen quasi-inverse $\phi^{-1}$.
With a slight abuse of notation, we denote by ${}_\TT\Sbb_{*}$ and ${}_\TT\Sigma_{*}$ the automorphisms of the stable category $\stmod\UU$ induced by ${}_\TT\Sbb$ and ${}_\TT\Sigma$ respectively.

\begin{proposition}\label{mod CY is CY}
  Let $\TT$ be a triangulated category, $\UU$ a $d\Z$-cluster-tilting subcategory of $\TT$, and $\ell\ge1$ and $m$ integers.
\begin{enumerate}[\rm(a)]
\item The category $\stmod\UU$ is triangulated, and satisfies
\[{}_{\stmod\UU}\Sbb={}_\TT\Sbb_{*}\circ {}_{\stmod\UU}\Sigma^{-1}\ \mbox{ and }\ {}_{\stmod\UU}\Sigma^{d+2}={}_\TT\Sigma_{*}^{d}.\]
\item If $\alpha:={}_\TT\Sbb^\ell \circ{}_\TT\Sigma^{-m}$ satisfies $\alpha(\UU)=\UU$, then 
\[{}_{\stmod\UU}\Sbb^{d \ell}={}_{\stmod\UU}\Sigma^{(d+2)m-d \ell}\circ\alpha_*^{d}.\]
\item If $\TT$ is $\frac{m}{\ell}$-Calabi--Yau, then $\stmod\UU$ is a $\frac{(d+2)m-d\ell}{d\ell}$-Calabi--Yau triangulated category.
\end{enumerate}
\end{proposition}

\begin{proof}
(a) See \cite[Proposition 4.2]{IO} and \cite[Proposition~4.4]{IO}.

(b) Since $\UU$ is $d$-cluster-tilting, the identity
${}_\TT\Sbb(\UU)=({}_\TT\Sigma^{-d}\circ{}_\TT\Sbb)(\UU)=\UU$ holds by \cite[Proposition~3.4]{IYo}.  
By (a), we have
\[{}_{\stmod\UU}\Sbb^{d\ell}={}_\TT\Sbb^{d\ell}_{*}\circ{}_{\stmod\UU}\Sigma^{-d\ell}={}_\TT\Sigma_{*}^{dm}\circ\alpha_*^d\circ{}_{\stmod\UU}\Sigma^{-d\ell}={}_{\stmod\UU}\Sigma^{(d+2)m-d\ell}\circ\alpha_*^d.\]

(c) This is immediate from (b).
\end{proof}

Endomorphism algebras of $d\Z$-cluster-tilting objects are a source of examples of twisted periodic algebras.

\begin{proposition}\label{d-CY tilted}
Let $\TT$ be a $k$-linear $\Hom$-finite triangulated category, $M\in\TT$ a $d\Z$-cluster-tilting object, and $E:=\End_{\TT}(M)$.
\begin{enumerate}[\rm(a)]
\item The algebra $E$ is twisted $(d+2)$-periodic.
\item Assume that $\TT$ is algebraic. If ${}_\TT\Sigma^{dr}\simeq 1$ as functors on $\add M$, then $E$ is $(d+2)r$-periodic.
\end{enumerate}
\end{proposition}

\begin{proof}
(a) This is immediate from Proposition \ref{mod CY is CY}(a).

(b) See \cite[Theorem 1.1]{Du3}.
\end{proof}

The following is a typical example of an application of Proposition~\ref{d-CY tilted}.
For a finite-dimensional algebra $A$ with $\gldim A\le d$, the \emph{$(d+1)$-preprojective algebra} of $A$ is defined as $\Pi:=T_A\Ext^d_A(DA,A)$, where $T_A$ denotes the tensor algebra. 

\begin{proposition}
Let $A$ be a $d$-representation-finite algebra with $\gldim A\le d$, and $\Pi$ the $(d+1)$-preprojective algebra of $A$.
Then $\Pi$ is twisted $(d+2)$-periodic.
If $A$ is $(m/\ell)$-Calabi--Yau, then $\Pi$ is $(d+2)r$-periodic for $r = d(d\ell-m)/\gcd(m,d)$.
\end{proposition}

\begin{proof}
Let $\CC_d(A)$ be the $d$-cluster category of $A$, that is, the triangulated hull of the orbit category $\Db(\mod A)/(\Sbb\circ[-d])$. Then $A$ is a $d\Z$-cluster-tilting object in $\CC_d(A)$ with $\End_{\CC_d(A)}(A)\simeq\Pi$.
Hence, $\Pi$ is twisted $(d+2)$-periodic, by Proposition~\ref{d-CY tilted}.

If $A$ is $(m/\ell)$-Calabi--Yau then $[d\ell-m] \simeq (\Sbb\circ[-d])^{-\ell}$ on $\Db(\mod A)$.
Let $F:\Db(\mod A)\to \CC_d(A)$ be the canonical functor. Then $\add A = F(\add A)\subset \CC_d(A)$, and hence $\Sigma^{d\ell -m} \simeq (\Sbb\circ\Sigma^{-d})^{-\ell} \simeq 1$ as functors on the full subcategory $\add A\subset\CC_d(A)$.
Thus, $\Pi$ is $(d+2)r$-periodic for $r = d(d\ell-m)/\gcd(m,d)$ by Proposition \ref{d-CY tilted}.
\end{proof}

Recall from Example~\ref{eg:fCYdimd} that any $d$-representation-finite algebra $A$ with $\gldim A\le d$ is twisted fractionally Calabi--Yau. Using this, we obtain the following result, which is an abelian analogue of Proposition \ref{d-CY tilted}.

\begin{theorem}\label{stable Auslander is CY}
Let $A$ be a $d$-representation-finite algebra with $\gldim A\le d$, $M$ the unique basic $d$-cluster-tilting $A$-module, and $E:=\underline{\End}_A(M)$ the stable $d$-Auslander algebra.
\begin{enumerate}[\rm(a)]
\item The algebra $E$
is twisted fractionally Calabi--Yau, and $T(E)$ is twisted periodic.
\item If $A$ is fractionally Calabi--Yau, then $E$ is fractionally Calabi--Yau, and $T(E)$ is periodic.
\end{enumerate}
\end{theorem}

\begin{proof}
By \cite[Theorem 4.7]{IO}, $\TT:=\Db(\mod A)$ has a $d$-cluster-tilting subcategory
\[\UU:=\add\{\nu_d^i(A)\mid i\in\Z\}\]
such that $\UU[d]=\UU$, and there is an equivalence $\UU\simeq\proj^{\Z}T(E)$. By \eqref{mod Z/nZ B} and \eqref{happel}, we have triangle equivalences $\Db(\mod E)\simeq\stmod^{\Z}T(E)\simeq\stmod\UU$.

(a) By our assumptions it follows that $A$ is twisted fractionally Calabi--Yau, and thus there exist $\ell,m\in\Z$, $\ell\ge1$, and $\psi\in\Aut_k(A)$ such that $\nu^\ell=[m]\circ\psi_*$. Possibly replacing $\ell$ and $m$ by multiples $a\ell$ and $am$ for some $a\in\Z$, we may assume that $\psi_*(P)\simeq P$ for each $P\in\proj A$. Then $\psi_*(X)\simeq X$ holds for all $X\in\UU$.
Take $\phi\in\Aut_k^{\Z}(T(E))$ such that $\psi_*:\UU\to\UU$ corresponds to $\phi_*:\proj^{\Z}T(E)\to\proj^{\Z}T(E)$ under the equivalence $\stmod^{\Z}T(E)\simeq\stmod\UU$.
Thus, setting $\varphi:=\phi|_E\in\Aut_k(E)$, we have a diagram
\[\xymatrix@R1em{\stmod\UU\ar[r]^\sim\ar[d]^{(\psi_*)_*}&\stmod^{\Z}T(E)\ar[d]^{\phi_*}&\Db(\mod E)\ar[d]^{\varphi_*}\ar[l]_{\sim}\\
\stmod\UU\ar[r]^\sim&\stmod^{\Z}T(E)&\Db(\mod E)\ar[l]_\sim
}\]
which commutes up to isomorphism of functors, and where the horizontal maps are triangle equivalences.

By Proposition~\ref{mod CY is CY}(b), 
\begin{align*}
  {}_{\stmod\UU}\Sbb^{d\ell}&={}_{\stmod\UU}\Sigma^{(d+2)m-d\ell}\circ(\psi_*)_*^{d} \quad\mbox{on}\quad \UU, \\
\intertext{which translates into}
\nu^{d\ell}&=[(d+2)m-d\ell]\circ\varphi_*^{d}\quad\mbox{on}\quad \Db(\mod E).
\end{align*}
This means that the algebra $E$ is twisted fractionally Calabi--Yau, and thus $T(E)$ is twisted periodic.

(b) The assertion follows from the argument above, where $\psi$, $\phi$ and $\varphi$ are specialized to the identity.
\end{proof}

Applying Theorem~\ref{stable Auslander is CY} to the path algebras of Dynkin type yields the following.
\begin{example}
Let $Q$ be a Dynkin quiver with Coxeter number $h$. Then the stable Auslander algebra $\Lambda$ of the path algebra $kQ$ is $\frac{2h-6}{h}$-Calabi--Yau.
We note that this result was also obtained in \cite{Lad} for $Q \neq A_{2n}$.
By Theorem~\ref{trivial extension is periodic}, the trivial extension algebra $T(\Lambda)$ is $6(h-2)$-periodic.
\end{example}

\subsection{Examples from tensor products and geometry} \label{example of CY}

We shall use tensor products of algebras of Dynkin type to construct families of algebras of unbounded Calabi--Yau dimensions and, thus, corresponding trivial extension algebras with unbounded minimal periods. For this, we need the following refinement of a result in \cite{HI}.

\begin{proposition} \label{prop:cydim}
Let $A_1,\ldots,A_t$ be fractionally Calabi--Yau Iwanaga--Gorenstein $k$-algebras, such that $A_i/\rad A_i$ is separable and $\CYdim A_i = (m_i,\ell_i)$ for each $i$, and 
\begin{equation} \label{l-m}
\ell=\lcm(\ell_1,...,\ell_t)\;, \quad m=\ell\left(\frac{m_1}{\ell_1}+\cdots+ \frac{m_t}{\ell_t}\right).
\end{equation}
\begin{enumerate}[\rm(a)]
\item The algebra $A:=\bigotimes_{i=1}^t A_i$ is Iwanaga--Gorenstein and $\frac{m}{\ell}$-Calabi--Yau.
 \label{ex:tensorprod}
\item If $A$ is ring-indecomposable, then $\CYdim A = (m, \ell)$.
\end{enumerate}
\end{proposition}

We first prove the following lemma.

\begin{lemma} \label{lma:cydim}
  For all $i\in\{1,\ldots, t\}$, let $A_i$ be a $k$-algebra, and $X_i,Y_i\in\DDD^{\rm b}(\mod A_i)$ complexes such that
  $\bigotimes_{i=1}^t X_i\simeq \bigotimes_{i=1}^t Y_i$ is indecomposable in $\DDD^{\rm b} (\mod (A_1\otimes_k\cdots \otimes_k A_t))$.
  Then there exist $\ell_1,\ldots,\ell_t\in\Z$ such that $X_i\simeq Y_i[\ell_i]$ for all $i$, and $\sum_{i=1}^t\ell_i = 0$.
\end{lemma}

\begin{proof}
For ease of notation, set $A=A_1\otimes_k\ldots\otimes_k A_t$. First, observe that
\[
\Hom_{\DDD^{\rm b}(\mod A)}\left(\bigotimes_{i=1}^t X_i, \bigotimes_{i=1}^t Y_i\right) = \bigoplus_{\substack{\ell_1,\ldots,\ell_t\in\Z \\ \sum_i\ell_i=0}}\, \bigotimes_{i=1}^t\Hom_{\DDD^{\rm b}(\mod A_i)}(X_i, Y_i[\ell_i])
\]
and hence any morphism  $f:\bigotimes_{i=1}^t X_i \to \bigotimes_{i=1}^t Y_i$ can be written as
\[
f = \sum_{r=1}^R f_1^{(r)}\otimes\ldots\otimes f_t^{(r)}\,, \quad\mbox{where}\quad
f_i^{(r)}:X_i\to Y_i[\ell_i^{(r)}] \quad \mbox{for some}\quad \ell_i^{(r)}\in\Z, 
\]
subject to the condition $\sum_{i=1}^t \ell_i^{(r)} =0$ for each $r\in\{1,\ldots,R\}$.

Assume that $f = \sum_{r=1}^R f_1^{(r)}\otimes\ldots\otimes f_t^{(r)}$ is an isomorphism. Since $\bigotimes_{i=1}^tX_i$ is indecomposable, $\End_{\DDD^{\rm b}(\mod A)}\left( \bigotimes_{i=1}^tX_i\right)$ is a local ring, and thus it follows that $f_1^{(r)}\otimes\ldots\otimes f_t^{(r)}$ must be an isomorphism for some $r$. 
But then $f_i^{(r)}:X_i\to Y_i[\ell_i^{(r)}]$ is an isomorphism for each $i$ and, as $\sum_{i=1}^t\ell_i^{(r)} = 0$, this proves the assertion in the lemma.
\end{proof}

\begin{proof}[Proof of Proposition~\ref{prop:cydim}]
(a) This is \cite[Proposition 1.4]{HI}.

(b) Let $a$ and $b$ be integers such that $A$ is $\frac{b}{a}$-Calabi--Yau. Since $A$ is $\frac{m}{\ell}$-Calabi--Yau by (a), it suffices to show that $\ell$ divides $a$. 
We consider the algebras $A_i$ as objects in $\DDD^{\rm b}(\mod A_i^{\mathrm{e}})$.
By Proposition~\ref{characterise CY}\eqref{fcy}, the Calabi--Yau property of $A$ gives an isomorphism
 \[
 A_1 \otimes_k\cdots \otimes_k A_{t-1} \otimes_k A_t [b] = A[b] \simeq (DA)^{\Lotimes_A a} \simeq
(DA_1)^{\Lotimes_{A_1} a}\otimes_k\cdots\otimes_k (DA_t)^{\Lotimes_{A_t}a}
  \]
  in $\Db(\mod A^{\mathrm{e}})$,
whence Lemma~\ref{lma:cydim} implies the existence of integers $n_1,\ldots, n_t\in\Z$ such that $A_i[n_i]\simeq (DA)^{\Lotimes_A a}$ for each $i$. Thus $A_i$ is $(n_i/a)$-Calabi--Yau. Since $\CYdim A_i = (m_i,\ell_i)$ it follows that $\ell_i$ divides $a$ for each $i$ and, consequently, so does $\ell = \lcm(\ell_1,\ldots,\ell_t)$.
\end{proof}

Combining Example~\ref{ex:dynkin} and Proposition~\ref{prop:cydim} above with our main result, Theorem~\ref{trivial extension is periodic}, we get the following corollary.

\begin{corollary} \label{periodoftp}
  Let $A = (kQ)^{\otimes t}$, where $Q$ is a quiver of Dynkin type and $t$ a positive integer. Then the minimal period of the trivial extension algebra $T(A)$ is
  \[ 
  \begin{cases}
    2((h-2)t + h) & \mbox{if ${\rm char} k \ne2$, $Q$ is of type $A_n$, and $t$ and $n$ are even;} \\
    ((h-2)t + h)/2 & \mbox{if ${\rm char} k = 2$, and $Q$ is of type $A_1$, $D_n$ with $n$ even, $E_7$ or $E_8$;} \\
    (h-2)t+h & \mbox{otherwise.}
  \end{cases}
  \]
\end{corollary}

\begin{proof}
  By Example~\ref{ex:dynkin} and Proposition~\ref{prop:cydim},
  $\CYdim\left((kQ)^{\otimes t}\right) = (t(h-2)/2, h/2)$ if $Q$ is of type $A_1$, $D_n$ with $n$ even, $E_7$ or $E_8$, and $\CYdim\left((kQ)^{\otimes t}\right) = (t(h-2), h)$ otherwise.
  The result now follows, by a straightforward calculation, from Theorem~\ref{trivial extension is periodic}.
\end{proof}

We remark that, except for the case $Q=A_1$, the algebra $A$ in Corollary~\ref{periodoftp} is wild (for $k$ algebraically closed) whenever $t\ge4$ \cite[Proposition~2.1(a)]{Le}. 
Then $T(A)$ is also wild, since $A$ is a quotient algebra of $T(A)$.
As mentioned in the introduction, the existence of a family of wild symmetric algebras with unbounded minimal periods appears to be previously unknown.
Note that for $Q$ of type $A_2$, the algebra $A$ is isomorphic to the incidence algebra of the Boolean lattice with $2^n$ elements.
In the next subsection, we will give more examples of fractionally Calabi--Yau incidence algebras.

As notion of `Calabi--Yau' originated in geometry, and the notion of `fractional Calabi--Yau' is a branch derived from it, we would like to mention a few examples that are related to algebraic geometry.
\begin{example}\label{eg:fCYdim3}
\begin{enumerate}[\rm(a)]
\item Geigle--Lenzing projective spaces \cite{HIMO} give a rich source of fractionally Calabi--Yau algebras of finite global dimension, called \emph{$d$-canonical algebras}. For $d=1$, they are the \emph{canonical algebras} associated with weighted projective lines \cite{GL}.
In fact, each $d$-canonical algebra of type $(p_1,\ldots,p_n)$ satisfying $n-d-1=\sum_{i=1}^n\frac{1}{p_i}$ is $\frac{dp}{p}$-Calabi--Yau for $p:=\lcm(p_1,\ldots,p_n)$. 
In the case $d=1$, there are 4 types: $(2,2,2,2)$, $(3,3,3)$, $(2,4,4)$ and $(2,3,6)$; see \cite{BES} and \cite{KLM} for different proofs of these cases.
In the case $d=2$, there are 18 types: $(2, 3, 7, 42)$, $(2, 3, 8, 24)$, $(2, 3, 9, 18)$, $(2, 3, 10, 15)$, $(2, 3, 12, 12)$, $(2, 4, 5, 20)$, $(2, 4, 6, 12)$, $(2, 4, 8, 8)$, $(2, 5, 5, 10)$, $(2, 6, 6, 6)$, $(3, 3, 4, 12)$, $(3, 3, 6, 6)$, $(3, 4, 4, 6)$, $(4, 4, 4, 4)$, $(2, 2, 2, 3, 6)$, $(2, 2, 2, 4, 4)$, $(2, 2, 3, 3, 3)$, $(2, 2, 2, 2, 2, 2)$.

There is another related source of fractionally Calabi--Yau algebras, called \emph{CM-canonical algebras} \cite{KLM,HIMO}, which appear in the study of singularity categories of Geigle-Lenzing hypersurfaces.

\item Additional examples, arising from algebraic geometry, of triangulated categories satisfying the fractional Calabi--Yau property, can be found in \cite{FK,K}.
\end{enumerate}
\end{example}

For $d$-canonical algebras, Theorem~\ref{trivial extension is periodic} gives the following result.

\begin{corollary}
Let $A$ be a $d$-canonical algebra of type $(p_1,\ldots,p_n)$ such that $n-d-1=\sum_{i=1}^n\frac{1}{p_i}$. Then $T(A)$ is $2(d+1)p$-periodic for $p:=\lcm(p_1,\ldots,p_n)$. 
\end{corollary}

\subsection{Examples from incidence algebras}\label{subsec:incidence}

We assume in the following that all posets are finite. 
A poset $P$ is said to be \emph{bounded} if it has a global maximum and a global minimum. Recall that the Hasse quiver $H_P$ of $P$ is the quiver whose vertices are elements of $P$ and arrows are the covering relations, i.e.\ $x\to y$ if $x<y$ and there is no other $z\in P$ with $x<z<y$.

\begin{definition} 
Let $P$ be a poset. The \emph{incidence algebra} $k[P]$ is the bound quiver algebra $kH_P/I$ where $I$ is the ideal generated by $\rho-\rho'$ for all pairs $(\rho,\rho')$ of parallel paths in $P$, i.e.\ paths with the same source $s(\rho)=s(\rho')$ and same target $t(\rho)=t(\rho')$.
\end{definition}

Note that incidence algebras of bounded posets have finite global dimension, as the quiver $H_P$ is directed.
It is then natural to ask which ones of these algebras are 
fractionally Calabi--Yau.  We have already shown in the previous subsection that when $P$ is the Boolean lattice, then $k[P]$ is fractionally Calabi--Yau.  Another fundamental example is the following one.

\begin{example}\label{eg:fCYdim2}
Chapoton conjectured in \cite{Cha} that the incidence algebra $k[T_n]$ of the $n$-th Tamari lattice $T_n$ is $\frac{n(n-1)}{2n+2}$-Calabi--Yau, and it was proved in \cite{R} by Rognerud.
In fact, Rognerud showed that for $n\geq 3$, $\CYdim(T_n)=(n(n-1), 2n+2)$; see \cite[Remark 8.4]{R}.
\end{example}

A related conjecture of Chapoton predicts that the incidence algebras of the distributive lattices of order ideals of the positive root posets from semisimple Lie algebras are fractionally Calabi--Yau, see \cite[Conjecture 5.3]{Cha} for a conjecture that would imply this for Dynkin type $A_n$ and \cite{Y} for a discussion of this conjecture for general Dynkin types. More generally, the fractional Calabi--Yau property seems to be deeply intertwined with the periodicity of the Coxeter transformations (c.f.{} Section~\ref{sec:DynkinCY}) which, in turn, is related to the notion of rowmotion in combinatorics, see for example \cite{MTY}.

Another connection of the fractional Calabi--Yau property of the incidence algebra with the properties of the associated poset is studied in \cite{DPW}.

\medskip 
We have seen that, for an algebra $A$ of finite global dimension, the fractional Calabi--Yau property is equivalent to periodicity of the trivial extension algebra $T(A)$. 
On the other hand, (twisted) periodicity of a symmetric algebra is something that can be checked using computer packages such as \cite{QPA}.
Thus, this opens a new approach to the aforementioned unpublished conjecture of Chapoton, as well as to the classification of fractionally Calabi--Yau algebras of finite global dimension in general.
To demonstrate our methods, we show in this subsection some new examples of incidence algebras that were not previously known to be fractionally Calabi--Yau.

Let us start by showing that, for many posets, the outer automorphism group is finite, and thus that the periodicity conjecture (Question \ref{conj:periodicity}) is true for the trivial extensions of the corresponding incidence algebras.

\begin{proposition}\cite[Corollary 7.3.7]{SO} \label{outerautoposet}
Let $P$ be a finite poset containing an element $x \in P$ that is comparable with any other element in $P$. Then any automorphism of the incidence algebra $k[P]$ of $P$ is the composition of an inner automorphism of $k[P]$ with an automorphism of $P$. In particular, the outer automorphism group of $k[P]$ is finite.
\end{proposition}

The following is immediate from Corollary~\ref{finite out} and Proposition~\ref{outerautoposet}.

\begin{theorem}\label{periodicity conjecture for poset}
Let $P$ be a poset containing an element $x \in P$ that is comparable to any other element in $P$. Then the periodicity conjecture is true for the trivial extension algebra $T(k[P])$.
\end{theorem}

This applies in particular to any bounded poset and thus to any lattice.

It is a routine exercise to calculate the explicit quiver with relations for the trivial extension of any $k[P]$, for example using \cite[Corollary 3.12]{FP}.
We sketch here a proof for the special case where the poset is bounded.

\begin{proposition} \label{quiverrelationsincidence}
  Let $P$ be a finite poset with distinct global maximum $\underline{1}$ and global minimum $\underline{0}$, and $I\lhd kH_P$ such that $k[P]\simeq kH_P/I$.
  \begin{enumerate}[\rm(a)]
  \item
    Let $H'_P$ be the quiver obtained from $H_P$ by adjoining a single arrow $(\underline{1} \xrightarrow{w} \underline{0})$.
  \item For any two vertices $a,b\in P$, denote by $p^a_b$ the unique (modulo $I$) non-trivial path from $a$ to $b$ in $H'_P$. Let $R\lhd k H_P'$ be the ideal generated by $I$ together with
    \[
    \left\{\alpha p_l^l,\:p_l^l\alpha \:\big{|}\: l\in P, \,\alpha\in(H_P')_1\right\} \:\cup\:
    \left\{p_{\underline{1}}^{a} w p_b^{\underline{0}} \:\big{|}\: a \not\le b \mbox{\rm{} and } b\not\le a\right\} .
    \]
  \end{enumerate}
  Then $T(k[P])\simeq kH_P'/R$.
\end{proposition}

\begin{proof}
Since there is a unique maximal path in $k[P]$, namely the one from $\underline{0}$ to $\underline{1}$, it follows from \cite[2.2, 2.4]{FP} that we only have to add one arrow $(\underline{1} \xrightarrow{w} \underline{0})$.
The relations are readily derived by applying \cite[Cor 3.12]{FP}.
\end{proof}

We will now look at several concrete examples obtained with the help of \cite{QPA} and \cite{SAGE}.
In the rest of this section, we will assume that the field has characteristic 0. We will restrict our attention to incidence algebras of distributive lattices, as these contain many important examples and have nicer homological properties compared with general posets -- see for example \cite{IM}.

Let us start with an example of a well studied poset whose trivial extension turns out to be periodic.
The study of free distributive lattices goes back to Dedekind \cite{De}, who studied related problems and posed the -- still open -- problem of finding an explicit formula for the number of elements of the free distributive lattice on $n$ generators.

\begin{example}\label{eg:free dist. lat}
Let $L$ be the free distributive lattice on 3 generators, that is, the distributive lattice of order ideals of the Boolean lattice of a 3-set. 
The Hasse quiver $H_L$ is of the form:
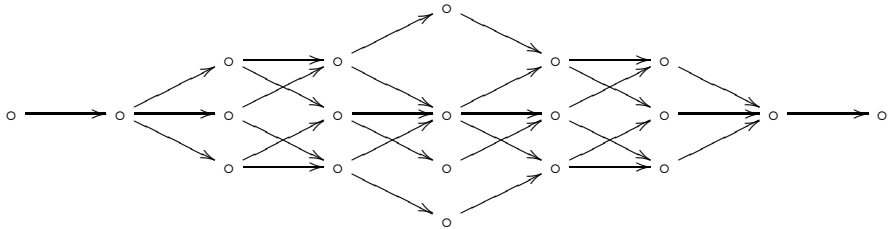
\[
  {\scriptsize \xymatrix@C=30pt@R=10pt{
&&&&\circ\ar[dr]\\
& & \circ\ar[r]\ar[rd] & \circ\ar[ru]\ar[rd] &  & \circ\ar[r]\ar[rd] & \circ\ar[rd] & \\
\circ\ar[r]&\circ\ar[ru]\ar[r]\ar[rd] & \circ\ar[ru]\ar[rd] & \circ\ar[r]\ar[rd] & \circ\ar[ru]\ar[r]\ar[rd]& \circ\ar[ru]\ar[rd] & \circ\ar[r] & \circ\ar[r]&\circ \\
& & \circ\ar[ru]\ar[r] & \circ\ar[rd]\ar[ru] & \circ\ar[ru] & \circ\ar[r]\ar[ru] & \circ\ar[ru] & & \\
& & & & \circ\ar[ru] & & & &
}}\]
Thus, the trivial extension algebra $T(k[L])$ is given by the following quiver, with relations as explained in Proposition~\ref{quiverrelationsincidence}.
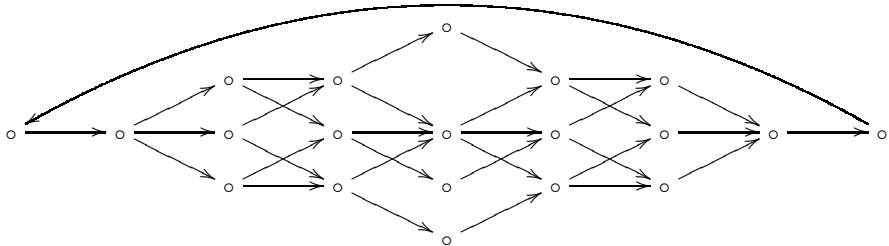
\[
{\scriptsize\xymatrix@C=30pt@R=10pt{
&&&&\circ\ar[dr]\\
& & \circ\ar[r]\ar[rd] & \circ\ar[ru]\ar[rd] &  & \circ\ar[r]\ar[rd] & \circ\ar[rd] & \\
\circ\ar[r]&\circ\ar[ru]\ar[r]\ar[rd] & \circ\ar[ru]\ar[rd] & \circ\ar[r]\ar[rd] & \circ\ar[ru]\ar[r]\ar[rd]& \circ\ar[ru]\ar[rd] & \circ\ar[r] & \circ\ar[r]&\circ\ar@/_4pc/[llllllll] \\
& & \circ\ar[ru]\ar[r] & \circ\ar[rd]\ar[ru] & \circ\ar[ru] & \circ\ar[r]\ar[ru] & \circ\ar[ru] & & \\
& & & & \circ\ar[ru] & & & &
}}\]
Using \cite{QPA}, we have verified that every that every simple module $S$ satisfies $\Omega^{14}_{T(k[L])}(S)\simeq S$; thus, $T(k[L])$ is twisted periodic. 
Theorem~\ref{periodicity conjecture for poset} now implies that $T(k[L])$ is periodic and hence, by Proposition~\ref{periodic to fCY}, the incidence algebra $k[L]$ is fractionally Calabi--Yau.
\end{example}

We demonstrate our methods by sketching a classification of all fractionally Calabi--Yau incidence algebras $k[L]$ for distributive lattices $L$ of size $11$, which gives new examples of fractionally Calabi--Yau algebras.
As in Example~\ref{eg:free dist. lat}, the example to follow was obtained using the GAP-package \cite{QPA}.

\begin{example}\label{11 points}
There are 82 distributive lattices of size 11, see \cite{oeis}.

It turns out that only 19 of these 82 distributive lattices have periodic Coxeter matrix, which is a necessary condition for an algebra of finite global dimension to be fractionally Calabi--Yau. For this calculation, we used \cite[Theorem 2.7]{KP}, which gives an upper bound of the period of $n \times n$ integer matrices.

For 15 of those 19 lattices $L$, computer calculations show that all simple modules of $T(k[L])$ are periodic. From Theorem~\ref{periodicity conjecture for poset}, it thus follows that these algebras are periodic and, consequently, that the incidence algebras $k[L]$ are fractionally Calabi--Yau.
For the trivial extensional algebras of the remaining four lattices, one can show that the dimensions of the syzygies of a given simple module go to infinity, hence these algebras are not periodic.

For explicit calculations and a list of all 15 incidence algebras that are fractionally Calabi--Yau, we refer to the forthcoming work \cite{M}, where a detailed classification of fractionally Calabi--Yau incidence algebras of posets of small cardinalities will be given.  Here, we just give two examples, discovered by the computer, of genuinely new fractionally Calabi--Yau algebras.
The first example is the distributive lattice $L$ with Hasse quiver:
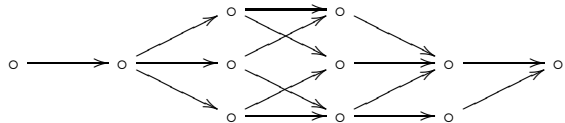
\[
\scriptsize{\xymatrix@C=30pt@R=10pt{
 &  & \circ \ar[r]\ar[rd] & \circ\ar[rd] &  &  \\
\circ\ar[r] & \circ\ar[ru]\ar[r]\ar[rd] & \circ\ar[ru]\ar[rd] & \circ\ar[r] & \circ\ar[r] & \circ \\
 & & \circ\ar[ru]\ar[r] & \circ\ar[ru]\ar[r] & \circ\ar[ru] &  
}}
\]
In this example all simple modules $S$ of the trivial extension $T(k[L])$ satisfy $\Omega^{38}(S) \simeq S$ and the Coxeter polynomial of $k[L]$ is equal to $ x^{11}+x^{10}+x^9+x^2+x+1$.

The second example has the following Hasse quiver:
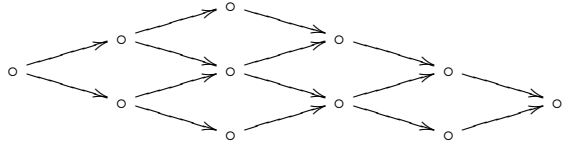
\[
{\scriptsize
\xymatrix@C=30pt@R=2pt{
 & & \circ\ar[rd] & & & \\
& \circ \ar[ru]\ar[rd]& &\circ\ar[rd] & & & \\
\circ \ar[ru]\ar[rd]& & \circ\ar[ru]\ar[rd] & & \circ\ar[rd] &  \\
& \circ\ar[ru]\ar[rd]& &\circ\ar[ru]\ar[rd] &  & \circ\\
 & & \circ\ar[ru] & & \circ\ar[ru] &
}}\]
All simple modules $S$ of the trivial extension $T(k[L])$ satisfy $\Omega^{31}(S) \simeq S$, and $k[L]$ has Coxeter polynomial $x^{11}+x^{10}-x^6-x^5+x+1$.

We remark that the incidence algebras of the two lattices above are not derived equivalent to Dynkin algebras. To see this, we compare their Coxeter polynomials (which is a derived invariant) against those of Dynkin type $A_{11}$ and $D_{11}$ to notice that they truly lie outside the derived equivalence class of Dynkin type path algebras.
\end{example}


\section{Appendix: Calabi--Yau dimension of Dynkin quivers}\label{sec:DynkinCY}

One fundamental class of fractionally Calabi--Yau algebras is given by the path algebras of Dynkin quivers.
The following result, which is a stronger version of \cite[0.3]{MY}, gives the minimal Calabi--Yau dimensions of these algebras.

\begin{theorem}\label{Dynkin is CY}
Let $Q$ be a Dynkin quiver, and $h$ the Coxeter number of the corresponding Dynkin type. Then
\[
\CYdim kQ=\begin{cases}
(\frac{h}{2}-1,\frac{h}{2}), & \text{if $Q$ is of type $A_1$, $D_n$ with $n$ even, $E_7$ or $E_8$;}\\
(h-2,h), & \text{otherwise.}
\end{cases}
\]
\end{theorem}

We give a simple direct proof based on elementary facts on quiver representations. Unlike \cite[0.3]{MY}, we do not need any explicit case-by-case calculations.

\begin{lemma}\label{h and h/2}
Let $X\in\Db(\mod kQ)$. Then $\tau^{-h}(X)\simeq X[2]$. If $Q$ is of type $A_1$, $D_n$ with $n$ even, $E_7$, or $E_8$, then $\tau^{-h/2}(X)\simeq X[1]$.
\end{lemma}

Although this can be deduced from the shape of the Auslander--Reiten quiver of $\Db(\mod kQ)$, we shall give a direct proof. Recall that $h$ is the order of the Coxeter transformation $c$. Moreover, $c^{h/2}=-1$ holds if and only if $Q$ is of type $A_1$, $D_n$ with $n$ even, $E_7$, or $E_8$.
The number of roots in the corresponding root system is $hn$, where $n=|Q_0|$.

\begin{proof}
Since $kQ$ is hereditary, each indecomposable object in $\Db(\mod kQ)$ is a shift of an indecomposable $kQ$-module. Thus we can assume that $X$ is an indecomposable $kQ$-module.

To prove the first part, recall that the action of $\tau$ on the Grothendieck group $K_0(\mod kQ)$ is given by the Coxeter transformation $c$.
Thus $[\tau^{-h}(X)] = [X]$ holds in $K_0(\mod kQ)$. By Gabriel's theorem \cite{Gab}, indecomposable $kQ$-modules are determined by their classes in $K_0(\mod kQ)$. Therefore, $\tau^{-h}(X)\simeq X[2a]$ holds for some $a\in\Z$. Here $a\ge0$, since $\gldim kQ\le 1$ implies $\tau^{-1}(\DDD^{\le0}(\mod kQ))\subset\DDD^{\le0}(\mod kQ)$, where $\DDD^{\le 0}(\mod kQ)$ is the aisle of the canonical $t$-structure of $\Db(\mod kQ)$.  Moreover, $a=0$ is not possible, since there are no periodic $\tau$-orbits. Since there are precisely $n=|Q_0|$ $\tau$-orbits, and $hn$ indecomposable objects in $\Db(\mod kQ)$ with non-trivial cohomology in degree $-1$ or $0$, a counting argument implies that $a\ge2$ is also impossible. Thus, $a=1$, that is, $\tau^{-h}(X)\simeq X[2]$.

We now prove the second claim.
Since $c^{h/2}=-1$ holds in this case, by looking at the class in $K_0(\mod kQ)$, we obtain $\tau^{h/2}(X) = X[2b+1]$ for some $b\in\Z$ which, together with the previous result, implies that $\tau^{-h/2}(X)\simeq X[1]$.
\end{proof}

\begin{proof}[Proof of Theorem \ref{Dynkin is CY}]
Let $(m,\ell):=(\frac{h}{2}-1,\frac{h}{2})$ when $Q$ is of type $A_1$, $D_n$ with even $n$ or $E_7, E_8$, and $(m,\ell):=(h-2,h)$ else.

By Lemma \ref{h and h/2}, for each $i\in Q_0$, we have $\tau^{-\ell}(e_ikQ)\simeq e_ikQ[\ell-m]$ or, equivalently, $\nu^{\ell}(e_ikQ)\simeq e_ikQ[m]$. Thus $\nu^{\ell}(kQ)\simeq kQ[m]$ and, by Proposition~\ref{characterise CY}(a), there exists $\phi\in\Aut_k(kQ)$ such that $D(kQ)^{\Lotimes\ell}\simeq {}_\phi(kQ)_1[m]$ in $\Db(\mod(kQ)^{\mathrm{e}})$. As explained in Section~\ref{section 1.1}, there exists an automorphism $\psi\in\Aut_k(kQ)$ that acts on $\{e_i\mid i\in Q_0\}$ and coincides with $\phi$ in $\Out_k(A)$. Then
\[e_ikQ[m]\simeq\nu^{\ell}(e_ikQ)=e_iD(kQ)^{\Lotimes\ell}\simeq\psi(e_i)kQ[m]\]
in $\Db(\mod kQ)$, and hence $\psi(e_i)=e_i$ for all $i\in Q_0$.
Since $Q$ is a tree, this implies that $\psi$ is an inner automorphism.
Thus $D(kQ)^{\Lotimes\ell}\simeq kQ[m]$ in $\Db(\mod(kQ)^{\mathrm{e}})$, whence $kQ$ is $\frac{m}{\ell}$-Calabi--Yau by Proposition~\ref{characterise CY}.

It remains to show that, if $kQ$ is $\frac{b}{a}$-Calabi--Yau for $a\ge1$, then $a$ is a multiple of $\ell$. Since $\tau^{a}\simeq[b-a]$ as functors on $\Db(\mod kQ)$, the action of $\tau^a = c^a$ on $K_0(\mod kQ)$ is $(-1)^{b-a}$. Thus $a$ is multiple of $\ell$, as desired.
\end{proof}

\section{Appendix 2: Testing fractionally Calabi-Yau with QPA}\label{sec:QPA}
In this section, we will very briefly explain how to use \cite{QPA} to verify example \ref{eg:free dist. lat} and example \ref{11 points}.
First we look at example \ref{eg:free dist. lat}. To display the Hasse diagram of the free distributive lattice on 3 symbols, you can use \cite{SAGE} and enter the following:
\begin{verbatim}
P=posets.BooleanLattice(3)
L=P.order_ideals_lattice().relabel()
plot(L)
\end{verbatim}
One can then translate the Hasse quiver into input for QPA to obtain the incidence algebra $A$.
This looks as follows:
\begin{tiny}
\begin{verbatim}
L0:=["x0", "x1", "x2", "x3", "x6", "x8", "x4", "x5", "x11", "x7", "x9",
 "x10", "x12", "x17", "x13", "x14", "x15", "x16", "x18", "x19"];
L1:=[["x0", "x1", "x0_x1"], ["x1", "x2", "x1_x2"], ["x1", "x3", "x1_x3"],
 ["x1", "x4", "x1_x4"], ["x2", "x6", "x2_x6"], ["x2", "x5", "x2_x5"], ["x3",
 "x6", "x3_x6"], ["x3", "x7", "x3_x7"], ["x6", "x8", "x6_x8"], ["x6", "x9",
 "x6_x9"], ["x8", "x10", "x8_x10"], ["x4", "x5", "x4_x5"], ["x4", "x7",
  "x4_x7"], ["x5", "x11", "x5_x11"], ["x5", "x9", "x5_x9"], ["x11", "x12",
   "x11_x12"], ["x7", "x9", "x7_x9"], ["x7", "x13", "x7_x13"], ["x9", "x10",
  "x9_x10"], ["x9", "x12", "x9_x12"], ["x9", "x14", "x9_x14"], ["x10",
   "x16", "x10_x16"], ["x10", "x17", "x10_x17"], ["x12", "x15", "x12_x15"],
   ["x12", "x17", "x12_x17"], ["x17", "x18", "x17_x18"], ["x13", "x14",
    "x13_x14"], ["x14", "x15", "x14_x15"], ["x14", "x16", "x14_x16"],
     ["x15", "x18", "x15_x18"], ["x16", "x18", "x16_x18"], ["x18", "x19",
     "x18_x19"]];


Q:=Quiver(L0,L1);KQ:=PathAlgebra(Rationals,Q);AssignGeneratorVariables(KQ);

rel:=[x1_x2*x2_x5 - x1_x4*x4_x5,
 x1_x2*x2_x6 - x1_x3*x3_x6,
 x1_x3*x3_x7 - x1_x4*x4_x7,
 -x2_x6*x6_x9 + x2_x5*x5_x9,
 x3_x6*x6_x9 - x3_x7*x7_x9,
 x6_x8*x8_x10 - x6_x9*x9_x10,
 x4_x5*x5_x9 - x4_x7*x7_x9,
 -x5_x11*x11_x12 + x5_x9*x9_x12,
 x7_x9*x9_x14 - x7_x13*x13_x14,
 x9_x10*x10_x16 - x9_x14*x14_x16,
 x9_x10*x10_x17 - x9_x12*x12_x17,
 x9_x12*x12_x15 - x9_x14*x14_x15,
 -x10_x17*x17_x18 + x10_x16*x16_x18,
 x12_x17*x17_x18 - x12_x15*x15_x18,
 -x14_x15*x15_x18 + x14_x16*x16_x18];

A:=KQ/rel;Dimension(A);
\end{verbatim}
\end{tiny}
The following commands calculate the trivial extension algebra $B$ of $A$ and checks whether the simple $B$-modules have finite period bounded by 20 (via the command IsOmegaPeriodic(i,20)):
\begin{verbatim}
B:=TrivialExtensionOfQuiverAlgebra(A);simB:=SimpleModules(B);
W:=[];for i in simB do Append(W,[IsOmegaPeriodic(i,20)]);od;W;
\end{verbatim}
If the final list $W$ contains only integers, then one has shown that all simple $B$-modules are periodic and thus the algebra $A$ is fractionally Calabi-Yau by our results in section 8. After a long calculation, one sees that the list $W$ only contains twenty times the integer 14 and thus we have shown that $A$ is fractionally Calabi-Yau.
Now we briefly sketch how to see the results in example \ref{11 points}.
First we can use Sage as follows to filter out the distributive lattices on 11 points with Coxeter matrix of finite order:
\begin{verbatim}
n=11
posets = [p.with_bounds() for p in Posets(n-2)]
lattices=[p for p in posets if p.is_lattice()]
distlattices = [p for p in lattices if LatticePoset(p).is_distributive()]
coxeterperiodicdistlattices=[p for p in distlattices if
 p.coxeter_transformation().multiplicative_order() in ZZ]
display(len(distlattices))
display(coxeterperiodicdistlattices)
\end{verbatim}
We indeed get 19 such distributive lattices that we can then enter into QPA as in the previous example and check whether the simple modules for the trivial extension algebra are periodic.
If all simple modules are periodic we can conclude as before that the incidence algebra is fractionally Calabi-Yau.
If the command IsOmegaPeriodic($S,t$) for a simple module $S$ and a larger number $t$ does not give a finite integer after a long calculation this indicates that the simple module $S$ is not periodic. Verifying this can not be done with the computer and is rather complicated. We refer to the forthcoming work \cite{M} for details.

\section*{Acknowledgments} 
Throughout this project, we are deeply indebted to the power of the GAP-package \cite{QPA} and \cite{SAGE}. We thank Joseph Grant for valuable discussions and for explaining to us his results in \cite{Gra}.
Part of this research was started during the Oberwolfach visit of AC and RM in January 2020, where Andrzej Skowro\'{n}ski's interest and enthusiasm in RM's first discovery of periodic wild symmetric algebras has encouraged us to pursue this work.  We are deeply saddened by his passing. 

AC was funded by JSPS Grant-in-Aid for Research Activity Start-up program 19K23401, and ED by JSPS Grant-in-Aid for Scientific Research (C) 18K03238.
OI was funded by JSPS Grant-in-Aid for Scientific Research (B) 16H03923, (C) 18K3209 and (S) 15H05738.
RM was funded by the DFG with the project number 428999796.
RM is thankful to Bernhard B\"ohmler and {\O}yvind Solberg for help with QPA and useful conversations.

\end{document}